\documentclass[11pt,twoside]{amsart}
\usepackage{verbatim,amscd,epsfig,amsmath,amsthm,amstext,amssymb,hyperref,extpfeil%,MnSymbol%tikz
}
\usepackage{tikz-cd}
\theoremstyle{plain}

\newtheorem{theorem}{Theorem}[section]

\newtheorem{Lemma}[theorem]{Lemma}
\newtheorem{lemma}[theorem]{Lemma}

\newtheorem{Corollary}[theorem]{Corollary}
\newtheorem{corollary}[theorem]{Corollary}

\newtheorem{remark}[theorem]{Remark}

\newtheorem{definition}[theorem]{Definition}

\newtheorem{Definition}[theorem]{Definition}
\newtheorem{Example}[theorem]{Example}

\newcommand{\unity}{{1\!\!\!\:\mathrm{l}}}

\newcommand{\C}{\mathbb{C}}

\newcommand{\Am}{\mathcal{A}}

\newcommand{\Sm}{\mathcal{S}}
\newcommand{\Tm}{\mathcal{T}}

\newcommand{\Sh}[1]{\mathcal{#1}} %%%%%%%% sheaves: O,L
\DeclareMathOperator*{\Res}{Res}

\DeclareMathOperator{\id}{id}

\newcommand{\ci}{i}
\newcommand{\dfx}{dx}
\newcommand{\dfy}{dy}

\begin{document}
\title{Universal Deformations of a Curve and a Differential}
\author[E. Carberry]{Emma Carberry}
\email{emma.carberry@sydney.edu.au}
\address{School of Mathematics and Statistics\\University of Sydney\\Australia}
\author[M. Schmidt]{Martin Ulrich Schmidt}
\email{schmidt@math.uni-mannheim.de}
\address{Mathematics Chair III\\
Universit\"at Mannheim\\
D-68131 Mannheim, Germany}

\begin{abstract}
We construct universal local deformations (Kuranishi families) for pairs consisting of a compact complex curve and a meromorphic 1-form. Each pair is assumed to be locally planar, a condition which in particular forces the periods of the meromorphic differential to be preserved by local deformations. The hyperelliptic case yields universal local deformations for the spectral data of integrable systems such as simply-periodic solutions of  the Korteweg-de Vries (KdV) equation or the sinh-Gordon equation (cylinders of constant mean curvature). % or the sine-Gordon equation (periodic pseudo-spherical surfaces).   
This is the first of two papers in which we shall develop a deformation theory of the spectral curve data of an integrable system.
% TODO: "periodic" is vague, discuss: okay or omit sine-Gordon?
\end{abstract} 
\date{\today}
\maketitle
%\input{introduction_single}
%\input{single}
%\input{existence}
%\input{hyperelliptic}
%%%%%%%%%%%%%%%%%%%%%%%%%%%%%%%%%%%%%%%%%%%%%%%%%%%%%%%%%%%%%%%%%%
%%%%%%%%%% introduction_single
%%%%%%%%%%%%%%%%%%%%%%%%%%%%%%%%%%%%%%%%%%%%%%%%%%%%%%%%%%%%%%%%%%
We construct universal local deformations of a locally planar pair $ (X,\dfx) $, where $ X $ is a compact one-dimensional reduced complex space and $\dfx $ is a meromorphic differential on $X$ with prescribed poles and no residues. The notation $\dfx $ emphasises the role of the local primitive and does not indicate that the differential is globally exact. We furthermore construct universal local deformations of hyperelliptic locally planar $ (X,\dfx) $, where hyperellipticity is required to be preserved under deformation.

Our interest in the deformations of a curve together with a differential arose from wanting to develop a deformation theory for the spectral data of an integrable system. Hyperelliptic curves with a single differential arise as the spectral data of a number of integrable systems, such as the KdV, sinh-Gordon, sine-Gordon and non-linear Schr\"{o}dinger (NLS) equations. As such, in the present paper we for example produce universal local deformations for simply-periodic solutions of the KdV equation, for constant mean curvature cylinders in $S ^ 3,\,\mathbb R ^ 3, H ^ 3 $ or $  S ^ 2\times\mathbb R $ (sinh-Gordon) and for pseudo-spherical surfaces (sine-Gordon) and vortex filaments (NLS) satisfying periodicity conditions. %Which are described by hyperelliptic spectral curve data $ (X,\dfx,\sigma) $.
One application of our results is that they provide a natural space, namely the base space of a universal deformation, on which  one can  integrate Whitham vector fields. Hence they enable us to construct Whitham deformations of spectral curve data.

In complex analytic geometry, there are two prominent and distinct deformation theories, namely for compact spaces and for complex space germs. We shall see (in Theorem~\ref {equivalent deformation}) that the addition of a meromorphic differential $dx$ so that the pair $(X,dx)$ is locally planar causes these two theories to essentially coincide.

Spectral curves are most commonly endowed with two  meromorphic differentials. Moreover, in the integrable systems applications, the local primitives of these differentials provide an local embedding of the curve into $\mathbb{C}^2$. This was our motivation for the notion of local planarity, which we introduce below. In a subsequent paper we shall study deformations of locally planar triples $ (X, \dfx, \dfy) $, where $\dfy$ is a second locally exact, meromorphic differential and the local primitives $x,y$ provide a local embedding of the curve $X$ into $\mathbb{C}^2$. We do not have the analogue of Theorem~\ref {equivalent deformation} for the two-differential case. However we shall study this case by combining the results of this paper with Whitham theory to identify amongst the deformations of $ (X,\dfx) $ those which additionally preserve the periods of $\dfy$.

%Let us first recall the definition of a universal local deformation. 
%A local deformation is a germ of a deformation. A local deformation is  universal if for any given local deformation, there exists a unique holomorphic map germ from the base of the given deformation to the base of the universal deformation such that the given deformation is isomorphic to the pullback of the universal deformation. 

At a generic point of $X$, a local primitve $x$ of $\dfx$ provides a local embedding of $X$ into $\C$. This fails exactly at the singularities of $X$ and at the smooth roots and poles of $\dfx$.
We say that the pair $ (X,\dfx) $ is locally planar (Definition~\ref{def:locally planar}) when 
\begin{enumerate}
\item at the singularities of $X$ and the smooth zeros of $\dfx$, any local primitive $ x $ of $\dfx $ can be complemented by another local function $ y $ to give a local embedding $ (x, y) $ of $ X $ into the affine plane $\C ^ 2 $,
\item at the poles of $\dfx$, an appropriate negative power of any primitive $x$ provides a local embedding into $\C$.
\end{enumerate}
%Away from the singularities of $X$ and the roots and poles of $\dfx$, the locally planar condition places no restriction on the pair $(X,\dfx)$. Indeed then $x$ is a local embedding of $X$ into $\mathbb{C}$, so can be complemented with any local function $y$ to obtain an embedding into $\mathbb{C}^2$. Hence it would be equivalent to require the existence of the complementary local functions $y$ only at the finitely many points where $X$ has a singularity or $\dfx$ has a root. However we find it simpler not to distinguish these points in our definition of local planarity. 

A deformation of a pair $(X,\dfx) $ is a deformation $X \to Y \twoheadrightarrow \Sm$ together with a  differential on $Y$ with prescribed poles which away from its poles has local primitives extending those of $\dfx $, see Definition~\ref{def2}. An important consequence of this additional assumption is  that the periods of $\dfx $ are preserved by local deformations of the data $(X,\dfx) $, as shown in Remark~\ref {Remark:periods}. This is particularly pertinent to the spectral curve applications, as then the periods of the differential $\dfx $ frequently satisfy restrictive conditions, such as being integer-valued. 
A morphism of deformations of $ (X,\dfx ) $ is a morphism of deformations of $ X $ which respects the local primitives on the total space, see Definition~\ref{def3}. 

Strikingly, such period-preserving deformations of compact curves with  differential $ (X,\dfx) $ can be studied by considering  deformations of germs $ (X_q, x_q) $ at just the finitely many points where $X$ has a singularity or $\dfx$ has a zero. These are the only points where the data $ (X_q, x_q) $ can deform non-trivially. We show in Theorem~\ref {equivalent deformation} that every deformation of the locally planar $ (X, \dfx) $ is uniquely obtained up to isomorphism by gluing the deformations of the germs  $(X_q,x_q)$ at these finitely many points together with a trivial deformation. Here the parameter space may need to be shrunk to effect the gluing; so it is convenient to work instead with local deformations, i.e. those for which the parameter space is a space germ. Thus the deformation theory of locally planar $ (X,\dfx) $ reduces to that of deformations of finitely many germs $ (X_q, x_q) $. This substantially simplifies the deformation theory of locally planar $ (X,\dfx) $ in comparison with that of the curve alone.  The analogous result does not hold for the Kodaira-Spencer theory of deformations of a curve.
In our apporach to deformations of locally planar $ (X,\dfx) $, it is  not necessary to create a global theory along the lines of \cite{ACG:11}. Instead we study deformations of such $ (X,\dfx) $ by adapting the deformation theory of plane curve singularities~\cite{dJP:00,GLS:07}. 

%TODO: Move the following paragraph to the third paper. 
%This approach is furthermore well adapted to the spectral curve examples. In the finite-type situation, the Bloch curve has infinite arithmetic genus but finite geometric genus. Furthermore the deformations permitted for it to remain a spectral curve of finite type are non-trivial only in the neighbourhood of finitely many points. Such deformations are uniquely determined by the deformations of finitely many germs, by patching arguments similar to the proof of Theorem~\ref {equivalent deformation}. Hence the deformation theory of this paper applies to  hyperelliptic Bloch curves of finite type.

Even in comparison with the deformation theory of plane curve singularities, that of locally planar $ (X, \dfx) $ enjoys  certain advantages. For plane curve singularities, using the Tjurina algebra of isomorphism classes of infinitesimal deformations, one can construct   semi-universal local deformations \cite{dJP:00,GLS:07}. 
 One cannot hope for  universal local deformations  as the dimension of the Tjurina algebra is not locally constant. The deformations of a plane curve singularity $X_q $ and of the locally planar pair $(X_q, x_q) $ coincide but for the latter, morphisms between deformations are required to respect the local functions $x_q $, Definition~\ref{definition:deformation-pairs}. This more restrictive notion
of morphisms of deformations results in the dimension of the space of isomorphism classes of infinitesimal deformations of $ (X,\dfx) $ being locally constant. Furthermore, the infinitesimal deformations form a vector bundle over the base (Theorem~\ref {l3}). 
%In fact, by our assumption that $(X,\dfx)$ is locally planar, $X$ is locally isomorphic at each of the aforementioned points \,$q$\, to the zero set $V(f)$ of some \,$f \in \mathbb{C}\{x,y\}$\,. We will see that the space of infinitesimal deformations of $(X, \dfx)$ is equal to the similar quotient $\sum_q \mathbb{C}\{x,y\}/\langle f_q,\frac{\partial f}{\partial y}\rangle$ and that the dimension of this quotient is locally constant. 
We shall use this vector bundle structure to show that any basis for the space of isomorphism classes of infinitesimal deformations of $ (X,\dfx) $ induces a universal local deformation \eqref {eq:universal 1}. %The vector bundle structure helps us to show  that these local deformations   \eqref {eq:universal 1} % Definition~\ref {universal1}  are universal (Theorem~\ref {t4}).
The construction of semi-universal deformations of plane curve singularities referenced above uses Grauert's powerful Approximation Theorem \cite[Chapter 8.2]{dJP:00}. The authors note that this can possibly be avoided, but we are not aware of a reference where this has been carried out. Our arguments do not invoke Grauert's Approximation Theorem.

% One may  ask whether the base space $\Tm $ of the universal deformation locally parameterises isomorphism classes of pairs $ (X, \dfx) $. It is a general property of universal deformations that this cannot be the case whenever $ (X, \dfx) $ possesses a non-trivial automorphism (Example~\ref {e1}), as this extends to an isomorphism of fibres over distinct elements of the base $\Tm $. We provide also a heuristic argument suggesting that such isomorphic fibres exist only when the base fibre has a non-trivial automorphism.

When the curve $ X $ is hyperelliptic, %(see section~\ref{sec:def1a}) 
then a local primitive of $ dx $ may be complemented by % a complementary function $ y $ (as in the definition of ``locally planar'') is provided by 
a local parameter $ y $ for the quotient of $ X $ by its hyperelliptic involution $\sigma $ to locally embed $X $  in $\C ^2 $. Restricting the allowable deformations to those which preserve the hyperellipticity of the curve, we again construct  universal local deformations for hyperelliptic $ (X,\dfx,\sigma) $.  This immediately yields universal local deformations for finite-gap solutions of any integrable system whose spectral data consists of a hyperelliptic curve and a differential, as above. We remark that this can be considered as a special case of the more typical situation in which a spectral curve has a pair of meromorphic differentials.

The objects, deformations and morphisms we shall consider are defined in Section~\ref{sec:def1}. Throughout the paper we incorporate the option of the data satisfying a reality condition, as is typical in the integrable systems applications. Section~\ref {sec:patching} explains that deformations of a curve $ X $ and differential $ \dfx $ as above are both generated and uniquely determined up to isomorphism by deformations of finitely many germs  of such data. In Section~\ref{sec:infinitesimal} we show that any basis for the space of isomorphism classes of infinitesimal deformations of $ (X,\dfx) $ induces an $x$-deformation \eqref {eq:universal representative} of $(X,\dfx)$. Importantly, the infinitesimal deformations of the fibres of \eqref {eq:universal representative} form a vector bundle over the base. Our main result, Theorem~\ref {t4}, is that the local $x$-deformations represented by \eqref {eq:universal representative} are universal. To prove this we utilise a decomposition of certain holomorphic functions, which is the subject of Section~\ref{sec:decomposition}.  Section~\ref {sec:existence} completes the proof of Theorem~\ref {t4}. It also contains a discussion of the role of automorphisms of the base fibre $ (X, \dfx) $. We conclude in Section~\ref{sec:def1a} by extending Theorem~\ref {t4} to the case when $ X $ and its deformations are hyperelliptic. This yields applications in integrable systems theory. %, for example for the sinh-Gordon equation or the KdV equation. 
\section{Deformations of pairs $(X,\dfx)$}\label{sec:def1}
In this section we describe the data and the deformations we shall consider. Our primary objects of study are local deformations of pairs $(X,\dfx)$, where $X$ is a compact one-dimensional complex space and $ \dfx $ is a locally exact meromorphic differential on $X$ with prescribed poles.  Throughout we shall use $x$ to denote a local primitive of $\dfx$. We do not assume that the primitive $x$ is globally defined. All deformations of such pairs $(X,\dfx)$ (see Definition~\ref {def2}) preserve the periods of $\dfx $ (Remark~\ref{Remark:periods}). In our main Theorem~\ref{t4} we shall assume that $(X,\dfx)$ is locally planar (see Definition~\ref{def:locally planar}), which  will have the pleasant consequence that such pairs are locally parameterised by a smooth manifold. 

We begin by recalling the basic notions of deformations of a complex space, which we will extend to define appropriate deformations of the data $(X,\dfx)$. A \emph{deformation } $X\hookrightarrow Y\twoheadrightarrow\Sm$ of a complex space $X$ is a pair of  complex spaces  $Y$ and $\Sm$ together with  a flat holomorphic map $Y\twoheadrightarrow\Sm$, a marked point $s_0\in\Sm$ and an isomorphism from $ X $ to the preimage of the point $s_0$ in $Y$. The space $X$ is called the  \emph{special fibre}, $Y$ the  \emph{total space} and $\Sm$ the  \emph{base space} of the deformation. 
% Flatness ensures that the fibres of the map $Y\twoheadrightarrow\Sm$ depend in a regular way on the points in $\Sm$. 
If \,$X$\, is compact, we further require that the map $Y\twoheadrightarrow\Sm$ is proper (c.f.\ \cite[Chapter~XI \S2]{ACG:11}. A  \emph{morphism} from a deformation $X\hookrightarrow Y\twoheadrightarrow\Sm\ni s_0$ to a deformation $X\hookrightarrow Z\twoheadrightarrow\Tm\ni t_0$ is a commutative diagram of holomorphic maps
$$ \begin{array}{cclcl}
X&\hookrightarrow&Y&\twoheadrightarrow&\Sm\ni s_0\\
\|&&\downarrow&&\downarrow\hspace{6mm}\downarrow\\
X&\hookrightarrow&Z&\twoheadrightarrow&\Tm\ni t_0
\end{array}$$
such that the surjective horizontal maps are flat. %Two deformations are  \emph{isomorphic} if there exist morphisms between them which are inverse to each other.
For any deformation $X\hookrightarrow Y\twoheadrightarrow\Sm\ni s_0$ and any open neighbourhood $O\subset\Sm$ of $s_0$ let $U\subset Y\twoheadrightarrow\Sm$ be the preimage of $O$. The resulting deformation $X\hookrightarrow U\twoheadrightarrow O\ni s_0$ is called the  \emph{restriction} of $X\hookrightarrow Y\twoheadrightarrow\Sm\ni s_0$ to $O$. To avoid double subscripts we shall denote the base point as $ 0 $; frequently $\mathcal S\subset\mathbb C ^ r $ and then without loss of generality we take the base point to be the origin.
\begin {definition}\label {definition:local}
A \emph{local deformation} $X\hookrightarrow Y\twoheadrightarrow\Sm_0 $ % about $ 0\in\mathcal S $ 
is an equivalence class of deformations $X\hookrightarrow Y\twoheadrightarrow\Sm$, each defined on a neighbourhood of $ 0\in\Sm $, where two such deformations are equivalent if their restrictions to some neighbourhood of $ 0\in\Sm $ are equal. \end {definition}

We now define the global data whose deformations we shall consider. Our interest in such objects is chiefly motivated by their appearance as the spectral curve data of an integrable system.
Recall that for $p\in X$,  % $X_p$ denote the space germ of $X$ at $p$ and 
$\dfx_p$ denotes the germ of the 1-form $\dfx$ at $p$. More generally we denote the germ of a complex space or of a map at a point $p$ by adorning the symbol by a subscript $p$. 
% Analogously $\mathbb{C}_0$ denotes the space germ of $\mathbb{C}$ at $0\in\mathbb{C}$. 
%
Henceforth, whenever we refer to a pair $(X,\dfx)$, the following is assumed:
\begin{definition}\label{definition:single}
We shall denote by $(X,\dfx)$ the data of a compact, one-dimensional and reduced complex space $X$ with finitely many marked points $p_1,\ldots,p_K$ and a meromorphic 1-form $\dfx$ satisfying the following conditions: %We say that $ (X,p_1,\ldots ,p_K, \dfx) $ or more briefly $ (X, \dfx) $ is  \emph{locally planar  with prescribed poles of order} $ M_{p_k}+1 $ at  $ p_k $, $ k = 1,\ldots , K $ if
\begin{itemize}
\item[{\bf(a)}]{\bf Prescribed poles:} For each $p\in\{p_1,\ldots,p_K\}$ there exists $u_p\in\Sh{O}_{X,p}$ which vanishes at $p$ and maps $X_p$ biregularly onto $\mathbb{C}_0$ such that $\dfx_p=d(u_p^{-M_p})$ for some $M_p\in \mathbb{Z}^+$.
\item[{\bf(b)}]{\bf $\mathbf { \dfx} $ is holomorphic on $X^\circ$:} For each $q\in X^\circ:=X\setminus\{p_1,\ldots,p_K\}$ there exists a germ $x_q\in\Sh{O}_{X,q}$ vanishing at $q$ with $d(x_q)=\dfx_q$. %For each $q\in X^\circ$ there exist $(x_q,y_q)\in\Sh{O}_{X,q}\times\Sh{O}_{X,q}$  which vanish at $q$ and map $X_q$ biregularly onto the zero set of some $f_q\in\mathbb{C}\{x,y\}$. Further we require that  $\dfx_q=d(x_q)$.
\end {itemize}
If in addition there exists an anti-holomorphic involution $\eta$ of $X$ such that the following condition holds, then $(X,\dfx)$ is called \emph{real with respect to $\eta$}.
\begin {itemize}
\item[{\bf(c)}]{\bf $\boldsymbol{\eta}$-real:} $\eta$ permutes $p_1,\ldots,p_K$ and acts as:
\begin{align}\label{eq:real}
%\eta^\ast \dfx=\Bar{\dfx}\quad
x_{\eta(q)}&=\eta^\ast\Bar{x}_q&%y_{\eta(q)}&=\eta^\ast\Bar{y}_q&
u_{\eta(p)}&=\eta^\ast\Bar{u}_p.%&f_{\eta(p)}(x,y)&=\Bar{f}_p(\Bar{x},\Bar{y}).
\end{align}
\end{itemize}
%For brevity we shall frequently refer only to the pair $ (X, \dfx) $, leaving the marked points implicit.
\end {definition}
We briefly mention some consequences of these conditions. Condition (a) guarantees that $p_1,\ldots,p_K$ are smooth points of $X$ and that $\dfx$ is locally exact about these poles. Hence $\dfx$ is locally exact on the 1-dimensional complex space $X$. Further, for each $p\in\{p_1,\ldots,p_K\}$ the equation $\dfx_p=d(u_p^{-M_p})$ determines the germ $u_p$ up to multiplication with a $M_p$-th root of unity. Finally, condition (c) implies that $\eta$ transforms $\dfx$ as $\eta^\ast \dfx=\overline{\dfx}$.

We now expand our notion of deformations to include the differential $\dfx$ and to include its prescribed poles at the smooth points $p_1,\ldots,p_K$:
\begin{definition}\label{def2}
An \emph{$x$-deformation} \begin{equation}\label{eq:deformation 1}
(X,\dfx)\xhookrightarrow{\iota}(Y,\dfx_Y)\xtwoheadrightarrow{\pi}\Sm
\end{equation}
of a pair $(X,\dfx) $ is given by a deformation $ X\xhookrightarrow{\iota}Y\xtwoheadrightarrow{\pi}\Sm $ of the complex space $X$ together with a meromorphic 1-form $\dfx_Y$ on $Y$, such that the following conditions hold:
\begin{itemize}
\item[{\bf(a)}]{\bf Prescribed poles:}
For each $p\in\{p_1,\ldots,p_K\}$, the germ $u_p$ of Definition~\ref{definition:single}~(a)   extends to $u_{Y,p}\in\Sh{O}_{Y,p}$ (i.e.~$u_p=u_{Y,p}\circ\iota$, where we identify $p\in X$ with $\iota(p)\in Y$).  Furthermore, $\dfx_{Y,p}=d(u_{Y,p}^{-M_p})$ holds, and $(u_{Y,p},\pi)$ maps $Y_p$ biregularly onto $(\mathbb{C}\times\Sm)_{(0,0)}$. Due to these conditions, each fibre $\pi^{-1}[\{s\}]$ of the deformation has marked points $p_1(s),\ldots,p_K(s)$, defined as the zeros of the $u_{Y,p_k}$. We write $Y^\circ=Y\setminus\{p_1(s),\ldots,p_K(s)\mid s\in\Sm\}$.
\item[{\bf(b)}] $\mathbf{\dfx_Y}$ {\bf is holomorphic and locally exact on \ }$\mathbf{Y^\circ}${\bf:} For any $\Tilde{q}\in Y^\circ$ there is a germ $x_{Y,\Tilde{q}}\in\Sh{O}_{Y,\Tilde{q}}$ vanishing at $\Tilde{q}$ with $d(x_{Y,\Tilde{q}})=\dfx_{Y,\Tilde{q}}$ and such that $x_q=x_{Y,q}\circ\iota$ for any $q\in X^\circ$.
Here we again identify $q\in X$ with $\iota(q) \in Y$.
%For each $q\in X^\circ=X\setminus\{p_1,\ldots,p_K\}$, the germs $x_q$ and $y_q$ of Definition~\ref{definition:single}~(b) extend to $x_{Y,q},y_{Y,q}\in\Sh{O}_{Y,q}$ such that $(x_{Y,q},y_{Y,q},s)$ maps $Y_q$ biregularly onto $V(F_q)$ with $F_q\in\mathbb{C}\{x,y\}\Hat{\otimes}\Sh{O}_{\Sm}$. 
%such that the germs $ x_{Y, q} $
%\item[{\bf(C1')}] On $Y_i\cap Y_j$ the maps $x_{Y_i}-x_{Y_j}$ are constant.
\end{itemize}
If in addition $(X,dx)$ is real with respect to $\eta$ and the following condition holds, then the $x$-deformation $(X,\dfx)\xhookrightarrow{\iota}(Y,\dfx_Y)\xtwoheadrightarrow{\pi}\Sm$ is called \emph{real with respect to $\eta$}. 
\begin{itemize}
\item[{\bf(c)}] {\bf $ {\boldsymbol{\eta}} $-real:} $\eta$ extends to an anti-holomorphic involution of $Y$ and induces an involution of $\Sm$ which commute with the maps $X\xhookrightarrow{\iota}Y\xtwoheadrightarrow{\pi}\Sm$ and act as
\begin{gather}\begin{aligned}\label{eq:real2}
x_{Y,\eta(q)}&=\eta^\ast\Bar{x}_{Y,q}&u_{Y,\eta(q)}&=\eta^\ast\Bar{u}_{Y,q}.
\end{aligned}\end{gather}
\end {itemize}
\end{definition}
Note that again (a) guarantees that $p_1,\ldots,p_K$ are deformed into smooth points of the fibres $\pi$ and for each $p\in\{p_1,\ldots,p_K\}$, $u_{Y,p}$ is uniquely determined by $\dfx_Y$ and $u_p$. The reality condition (c) implies $\eta^\ast \dfx_Y=\overline{\dfx}_Y$.

Furthermore, parts (a) and (b) of this definition together imply that $\dfx_Y$ is locally exact on all of \,$Y$\,. 
The local exactness of $\dfx_Y$ in Definition~\ref{def2} forces the periods of $\dfx$ to be locally preserved along fibres of the deformation, by a standard Stokes' theorem argument which we will now detail.

\begin{remark}\label {Remark:periods}
Let $(X,\dfx)\hookrightarrow(Y,\dfx_Y)\twoheadrightarrow\Sm$ be an $ x $-deformation of a pair $(X, \dfx)$. Suppose $s:(-\epsilon,\epsilon)\to\Sm$ and $\gamma:\mathbb{S}^1\times(-\epsilon,\epsilon)\to Y$ are continuous maps such that the diagram
$$\begin{array}{rcl}
\mathbb{S}^1\times(-\epsilon,\epsilon)&\xrightarrow{\gamma}& Y\\
\downarrow\hspace{3mm} &&\downarrow\\(-\epsilon,\epsilon)&\xrightarrow{s}&\Sm
\end{array}$$
commutes, where the left vertical map is the projection onto the second factor and $s(0)=0$. Then for $t\in(-\epsilon,\epsilon)$ the element of $\pi_1(Y)$ defined by $\gamma(\cdot,t)$ and the period $\int_{\gamma(\cdot,t)}\dfx_Y$ are independent of $t$. 
%For $ s\in S $, denote by $ X (s) $ the fibre of $ Y \twoheadrightarrow\Sm $ over $ s $. Fix a basis $\gamma_1,\ldots,\gamma_n $ of $ H_1 (X,\mathbb Z) $ and suppose that there is a continuous map 
%\begin {align*}
%S &\rightarrow Z_1 (X (s),\mathbb Z) ^ n\\
%s &\mapsto (\gamma_1 (s), \ldots, \gamma_n (s))
%\end {align*}
%so that the corresponding elements $\gamma_i (s) $ of $ H_1 (Y,\mathbb Z) $ are each constant. Then each period $\int_{\gamma_i (s)}\dfx_Y $ is independent of $s $. The analogous statement holds if we take instead a basis for the homology $ H_1 (X\setminus\{q_1,\ldots, q_m\}) $, where $ q_1,\ldots, q_m $ denotes the singularities of $ X $.
\end{remark}
\begin{proof}
The continuous map $\gamma$ provides a homotopy between $\gamma(\cdot,a)$ and $\gamma(\cdot,b)$ and hence the resulting element of $\pi_1(Y)$ is independent of $t$. For the second statement, we note that if $ Y $ and $\gamma$ are smooth and $\gamma$ avoids the poles of $\dfx_Y$ then the remark is a consequence of  Stokes' Theorem and the fact that $\dfx_Y$ is locally exact and hence closed. We adapt the proof of Stokes' Theorem and give an argument which does not require reformulation when these conditions fail.
Fix $a<b\in(-\epsilon,\epsilon)$.  We consider the map $\beta:[0,2\pi]\times[a,b]\to Y$ with $\beta(\varphi,t)=\gamma(e^{i\varphi},t)$. For each $(\varphi,t)\in[0,2\pi]\times[a,b]$ we may take an open rectangular neighbourhood $Q_{(\varphi,t)}\subset[0,2\pi]\times[a,b]$ such that the differential $\dfx_Y$ has a local integral $x_{Y,\beta(\varphi,t)}$ defined on $\beta(Q_{(\varphi,t)})$. If $\beta(\varphi,t)$ is not a pole of $\dfx_Y$ then this local function is the one guaranteed by condition~(b) of Definition~\ref{def2}. If $\beta(\varphi,t)$ is a pole then we take $x_{Y,\beta(\varphi,t)}=u_{Y,p}^{-M_p}$ as in condition~(a) of Definition~\ref{def2}. The $Q_{(\varphi,t)}$ form an open cover of $[0,2\pi]\times[a,b]$ and we take a finite subcover $Q_n$ with $n\in \{1,\dotsc,N\}$. On each $Q_n$ we have a function $\tilde{x}_n$ obtained by pulling back the corresponding local integral of $\dfx_Y$. For $ t\in [a, b] $ we may choose $ 0 = \varphi _0 <\varphi _1 <\cdots <\varphi _I = 2\pi $ so that each $ [\varphi _{i -1}, \varphi _i]\times\{t\}\subset Q_{n_i} $ for some $ n_i\in\{1, \ldots , N \} $. Hence
\begin{align*}
\int_{\gamma(\cdot,t)}\dfx_Y &=\sum_{i = 1} ^ I \left(\Tilde{x}_{n_i}(\varphi_i,t) - \Tilde{x}_{n_i} (\varphi_{i-1},t)\right)=\sum_{i=1}^I\left(\Tilde{x}_{n_i}(\varphi_i,t)-\Tilde{x}_{n_{i+1}}(\varphi_i,t)\right),
\end{align*}
with $\Tilde{x}_{n_{I+1}}:=\Tilde{x}_{n_1}$. Since $\Tilde{x}_{n_i}-\Tilde{x}_{n_{i+1}}$ is constant on $Q_{n_i}\cap Q_{n_{i+1}}$ the claim follows.
%HERE WE ARE
%
%Take $ q, q'\in X$ and continuous path $\beta $ in $X$ from $q $ to $q' $. If $\beta$ is contained in $X^\circ$ then the locally planar %condition of Definition~\ref {def2} gives that analytic continuation of $ x_{Y, q} $ along  $\beta $ defines a holomorphic function germ %$ x_{Y, q,\beta} $ at $q' $ which differs from $ x_{Y, q'} $ by the germ of a constant function on a neighbourhood of $ q' $ in $ Y $:
%\begin {equation}\label {equation:periods}
%x_{Y,q,\beta}-x_{Y,q'} = x_{Y,q,\beta}(q').
%\end {equation}
%This conclusion also holds if $\beta$ pases through poles of $\dfx_Y$
%Each period $\int_{\gamma_i (s)}\dfx_Y $ is calculated by summing finitely many such differences and so \eqref {equation:periods}  forces %the periods of $ \dfx_Y $ to be locally constant in $\Sm $.
%%, since $x_{Y,q}-x_{Y,q'}$ is locally constant in $Y$.
\end{proof}
Now we shall describe the morphisms of these deformations.
\begin{definition}\label{def3}
For any pair $(X,\dfx)$ a \emph{morphism} from one $x$-deformation $ (X,\dfx)\hookrightarrow(Y,\dfx_Y)\twoheadrightarrow\Sm $ to another $x$-deformation $(X,\dfx)\hookrightarrow(\tilde{Y},\dfx_{\tilde{Y}})\twoheadrightarrow\tilde{\Sm}$ is a pair $(\phi,\psi)$ of holomorphic maps such that the diagram
\begin{equation*}%\label{eq:morphism}
\begin{array}{cclcl}
X&\hookrightarrow&Y&\twoheadrightarrow&\Sm\\
\|&&\hspace{1mm}\downarrow\psi&&\downarrow\phi\\
X&\hookrightarrow&\tilde{Y}&\twoheadrightarrow&\tilde{\Sm}
\end{array}
\end{equation*}
commutes and the local functions of Definition~\ref{def2} satisfy
$u_{\tilde{Y},\psi(p)}\circ\psi_p=u_{Y,p}$ for each $p\in\{p_1,\ldots,p_K\}$ and $x_{\tilde{Y},\psi(q)}\circ\psi_q=x_{Y,q}$ for each $q\in Y^\circ $. 
\end {definition}
It immediately follows that $\dfx_Y$ is equal to $\psi^\ast \dfx_{\Tilde{Y}}$, and thus no special provision needs to be made for the differential in the above definition.

As for deformations of complex spaces, we have the following definition.
\begin{definition}
A \emph{local $x$-deformation}  $(X,\dfx)\hookrightarrow(Y,\dfx_Y)\twoheadrightarrow\Sm_0$ is an equivalence class of $x$-deformations $(X,\dfx)\hookrightarrow(Y,\dfx_Y)\twoheadrightarrow\Sm$ defined on a neighbourhood of $0$, where two such deformations are equivalent if their restrictions to some neighbourhood of $0$ are equal.
\end{definition}

The notions of morphism and of reality extend to local deformations in the obvious way.
Our primary interest lies in the category of  \emph{local deformations of locally planar $ (X,\dfx)$}. The main result of this paper is the construction of  universal objects in this category (Theorem~\ref{t4}) for locally planar pairs $(X,\dfx)$ (Definition~\ref{def:locally planar}). 

\section{Patching deformations}\label{sec:patching}
In this section we shall show that the deformation theory of our pairs $(X,\dfx)$ reduces to that of finitely many pairs $(X_q, x_q)$ of space germs $X_q$ and germs $x_q$. 

Let $(X,\dfx)$ be a pair as in Definition~\ref{definition:single} and $(X,\dfx)\hookrightarrow(Y,\dfx_Y)\xtwoheadrightarrow{\pi}\Sm$ an $x$-deformation. Let the finite set $\{q_1,\ldots,q_L\}$ consist of the singularities of $X$ together with the smooth roots of $\dfx$. The restriction of the corresponding local $x$-deformation to the complement of any open neighbourhood of $\{q_1,\ldots,q_L\}$ is isomorphic to the trivial local $x$-deformation. To see this, note first that condition (a) of Definition~\ref {def2} ensures that the space germs of $Y$ at the marked points $\{p_1, \ldots , p_K\} $ are isomorphic to the space germs of $\mathbb{C}\times\Sm$ at $(0,0)$. Furthermore, at smooth points $q\in X^\circ$ where $\dfx_q=d(x_q)$ does not vanish, $x_q$ maps $X_q$ biregularly onto $\mathbb{C}_0$ and $(x_{Y,q},\pi_q)$ maps $Y_q$ biregularly onto $(\mathbb{C}\times\Sm)_{(0,0)}$. This observation, which we formalise below, proves the uniqueness part of Theorem~\ref {equivalent deformation}. %Note that although this lemma is stated for locally planar $ (X, \dfx) $ with prescribed poles in order to avoid the introduction of more terminology, the fact that the curve may be locally embedded in $\C ^ 2 $ is not used here in any essential way.
%We extend the notion of deformations and their morphisms to the pairs $(X_q,x_q)$ with $q\in\{q_,\ldots,q_L\}$.
\begin{definition}\label{definition:deformation-pairs}
We consider pairs $(X_q,x_q)$ with $q\in X ^\circ $
% $q\in\{q_1,\ldots,q_L\}$ 
where $x_q\in\Sh{O}_{X_q}$ with $x_q(q)=0$. A  \emph {deformation} $(X_q,x_q)\xhookrightarrow{\iota}(Y_q,x_{Y,q})\twoheadrightarrow\Sm_0$ of $(X_q,x_q)$ is a deformation $X_q\xhookrightarrow{\iota}Y_q\twoheadrightarrow\Sm_0$ of space germs together with $x_{Y,q}\in\Sh{O}_{Y_q}$ such that $x_q=x_{Y,q}\circ\iota$.

A  \emph {morphism} from $(X_q,x_q)\hookrightarrow(Y_q,x_{Y,q})\twoheadrightarrow\Sm_0$ to $(X_q,x_q)\hookrightarrow(\Tilde{Y}_q,x_{\Tilde{Y},q})\twoheadrightarrow\Tilde{\Sm}_0$ is a pair $(\phi,\psi)$ of holomorphic maps
\begin{equation*}%\label{eq:morphism}
\begin{array}{cclcl}
X_q&\hookrightarrow&Y_q&\twoheadrightarrow&\Sm_0\\
\|&&\hspace{1mm}\downarrow\psi&&\downarrow\phi\\
X_q&\hookrightarrow&\Tilde{Y}_q&\twoheadrightarrow&\Tilde{\Sm}_0
\end{array}
\end{equation*}
which satisfies $x_{\Tilde{Y},q}=x_{Y,q}\circ\psi$.
\end{definition}

We note that an $x$-deformation $(X,dx)\hookrightarrow(Y,dx_Y)\twoheadrightarrow\Sm$ induces for each $q\in X ^  \circ $ a deformation $(X_q, x_q)\hookrightarrow (Y_q,x_{Y,q})\twoheadrightarrow\Sm_0$. 
% Furthermore, on these space germs we have the germs $x_q\in\Sh{O}_{X_q}$ and $x_{Y,q}\in\Sh{O}_{Y_q}$ and we notate this data as a deformation $(X_q,x_q)\hookrightarrow(Y_q,x_{Y,q})\twoheadrightarrow\Sm_0$.

%TODO: Change notation here and above from $x_q$ etc. to $x_l$, and move the introduction of the latter notation to the right place.

\begin{Lemma}\label {lemma:uniqueness}
For any pair $(X,\dfx)$ let $q_1,\ldots,q_L\in X$ denote the points at which either $X$ is not smooth or $\dfx$ has a smooth root. Let 
\begin{align}\label{deformations}
(X,\dfx)\hookrightarrow(Y,\dfx_Y)&\twoheadrightarrow\Sm,&
(X,\dfx)\hookrightarrow(\Tilde{Y},\dfx_{\Tilde{Y}})&\twoheadrightarrow\Tilde{\Sm}
\end{align}
be $x$-deformations such that for each $q\in\{q_1,\ldots,q_L\}$ the induced deformations $(X_q,x_q)\hookrightarrow(Y_q,x_{Y,q})\twoheadrightarrow\Sm_0$ and $(X_q,x_q)\hookrightarrow(\Tilde{Y}_q,x_{\Tilde{Y},q})\twoheadrightarrow\Tilde{\Sm}_0$ are isomorphic. % With biregular vertical maps
% $$\begin{array}{cclcl}
% X_q&\hookrightarrow&Y_q&\twoheadrightarrow&\Sm_0\\
%\|&&\hspace{1mm}\downarrow\psi_q&&\downarrow\phi_0\\
% X_q&\hookrightarrow&\Tilde{Y}_q&\twoheadrightarrow&\Tilde{\Sm}_0
%\end{array}$$
% such that $x_{Y,q}=x_{\Tilde{Y},q}\circ\psi_q$. 
Then the local $x$-deformations defined by \eqref{deformations} are isomorphic.
\end{Lemma}
\begin {proof}
For each of the marked points $p\in\{p_1,\ldots,p_K\}$, recall from Definition~\ref{def2} that  the maps $ (u_{Y, p},\pi_p): Y_p \rightarrow (\C\times\Sm)_{(0, 0)} $ and $ (u_{\Tilde{Y},p},\tilde{\pi}_p): \Tilde{Y}_p \rightarrow (\C\times\Tilde{\Sm})_{(0, 0)} $ are biregular. The germ $\psi_p:Y_p\to \Tilde{Y}_p$ defined by $\psi_p = (u_{\Tilde{Y},p},\tilde{\pi}_p)^ {- 1}\circ (\id_\C,\phi_0)\circ (u_{Y, p},\pi_p) $ satisfies $u_{\Tilde{Y},p}\circ\psi_p=u_{Y,p}$. Analogously for each $q\in X^\circ\setminus\{q_1,\ldots,q_L\}$, since $(x_{Y, q},\pi_q)$ and $(x_{\Tilde{Y}, q},\tilde{\pi}_q)$ are biregular, replacing $u$ by $x$ yields $\psi_q:Y_q\to \Tilde{Y}_q$ with $x_{\Tilde{Y},q}\circ\psi_q=x_{Y,q}$. All these germs $\psi_q:Y_q\to \Tilde{Y}_q$ with $q\in X$ fit together to form a holomorphic map $\psi$ from an open neighbourhood of the compact subset $X\subset Y$ into $\Tilde{Y}$. The complement of this open neighbourhood of $X$ is mapped by $Y\twoheadrightarrow\Sm$ onto the complement of an open neighbourhood $\Sm'$ of $0\in\Sm$, since a proper continuous map to a metrisable space is closed. We consider the restriction of the $x$-deformation $(X,\dfx)\hookrightarrow(Y,\dfx_Y)\twoheadrightarrow\Sm$ to $\Sm'$. The map $\psi$ defines a morphism from this restriction to  $(X,\dfx)\hookrightarrow(\Tilde{Y},\dfx_{\Tilde{Y}})\twoheadrightarrow\Tilde{\Sm}$. Reversing the roles of these two $ x $-deformations yields the inverse morphism. This shows uniqueness of~\eqref{eq:deformation 1} up to isomorphism and restriction to an open neighbourhood of $0$ in $\Sm$.\end{proof}

Our primary concern is with pairs $(X,\dfx)$ which are locally planar in the sense that for any $q\in X^\circ$ the space germ $X_q$ is biregular to a one-dimensional reduced space germ in  \,$\mathbb{C}^2$\, such that one of the coordinates is a primitive of $\dfx_q$. Away from the singularities of \,$X$\, and the smooth roots of $dx$, this condition is automatic, because then $x_q$ itself defines a biregular map to the space germ $\mathbb{C}_0$. Hence it suffices to require the condition of local planarity at the finitely many points of $X^\circ$ which are either singularities or smooth roots of $\dfx$.

\begin{Definition}\label{def:locally planar}
	For any $(X,\dfx)$ let $q_1,\ldots,q_L$ denote the points of $X^\circ$ which are either singularities of $X$ or smooth roots of $\dfx$. For $l=1,\ldots,L$, $(X_l,x_l)$ denotes the corresponding pair $(X_{q_l},x_{q_l})$. We call $(X,\dfx)$ \emph{locally planar} if for any $l=1,\dots,L$ there exists a \emph{complementary function} $y_l\in\Sh{O}_{X_l,q_l}$ which vanishes at $q_l$ such that $(x_l,y_l)$ maps $X_l$ biregularly onto the zero set $V(f_l)$ of some $f_l\in\mathbb{C}\{x,y\}$.
\end{Definition}

To avoid double indices we here and henceforth abbreviate the index $q_l$ for each $l=1,\ldots,L$ by the index $l$ so that we will hitherto write $X_l$ for the space germ $X_{q_l}$, $x_l=x_{q_l}$ and so forth. %Furthermore, we abbreviate the space germs $X_l=X_{q_l}$ and $Y_l=Y_{q_l}$.
For given $(X,\dfx)$ the germs $y_l$ are not unique. 

The pair $(X,\dfx)$ is locally planar if and only if for any $q\in X^\circ$ which is either a singularity or a smooth root of $\dfx$ there exists a germ $y_q\in\Sh{O}_{X,q}$ which vanishes at $q$, such that the algebra homomorphism $\mathbb{C}\{x_q,y_q\}\to\Sh{O}_{X,q}$ is surjective, or equivalently such that the stalk $\Sh{O}_{X,q}$ is generated by $y_q$ as a $\mathbb{C}\{x_q\}$-module.
In fact, due to the Local Parametrisation Theorem \cite[Theorem~3.4.14.]{dJP:00}, the kernel of this homomorphism is an ideal which is generated by one element $f_q\in\mathbb{C}\{x_q,y_q\}$.

An example where the local planarity condition fails is the triple point $\{(x,y,z)\in\mathbb{C}^3\mid xy=xz=yz=0\}$. The space germ at $(0,0,0)$ cannot be embedded into $\mathbb{C}^2$, because in the quotient ring $x$, $y$ and $z$ are linearly independent.

For a germ $\Sm_0 $ of a complex space $\Sm $ at $0\in\Sm$ we denote by $\mathbb{C}\{x,y\}\Hat{\otimes}\Sh{O}_{\Sm_0}$ the stalk of the holomorphic functions on $\mathbb{C}^2\times\Sm$ at $(x,y,s)=(0,0,0)$ (c.f.~\cite[Definition~7.3.6]{dJP:00}). It is proven in \cite[Corollary~II.1.6]{GLS:07} that all local deformations of a locally planar germ $V(f)$ preserve the locally planar condition in the following sense:
\begin{lemma}\label{le:embbede}
Any deformation of a space germ $V(f)$ with $f\in\mathbb{C}\{x,y\}$ is of the form
\[V(f)\hookrightarrow V(F)\twoheadrightarrow\Sm_0\]
with $F\in\mathbb{C}\{x,y\}\Hat{\otimes}\Sh{O}_{\Sm_0}$ and a unit $h\in\C\{x, y\}$ such that
\begin{equation}\label{eq:deform}
\hspace{34mm}F(x,y,0)=h(x,y)f(x,y).\hspace{32mm}\qed
\end{equation}
%Of the holomorphic functions on $\mathbb{C}^2\times\Sm$ at $(x,y,s)=(0,0,0)$.
\end{lemma}
%We also note that there is a slight abuse of terminology here in that whilst $X^\circ$ is locally planar with respect to a primitive of $\dfx $ and another coordinate, in fact $X$ does not satisfy this condition at the points $p_1,\ldots,p_K$. There it is instead locally biregular to an open set in $\mathbb{C}$.
A space germ $V(f)$ always has local germs $x_0$, $y_0$. In particular, by Lemma~\ref{le:embbede}, a deformation of $V(f)$ is automatically a deformation of the pair $(V(f),x_0)$, see Definition~\ref{definition:deformation-pairs}. The morphisms between deformations of $V(f)$ in which we are interested are morphisms between deformations of the pair $(V(f), x_0)$, which we term $x$-morphisms:
% This precisely reflects the additional structure that is imposed by $dx$ on the pair $(X,dx)$.

\begin {definition}\label{definition:deformation}
For any reduced space germ $V(f)$ with $f\in\mathbb{C}\{x,y\}$ an $x$-morphism from a deformation $V(f)\hookrightarrow V(F)\twoheadrightarrow\Sm_0$ to a deformation $V(f)\hookrightarrow V(\Tilde{F})\twoheadrightarrow\widetilde\Sm_{0}$ is a morphism  of deformations
\begin{gather}\label{eq:def germ}
\begin{array}{ccccl}
V(f)&\hookrightarrow& V(F) &\twoheadrightarrow&\Sm_0\\
\| &&\downarrow\psi&&\downarrow\phi\\
V(f)&\hookrightarrow&V(\Tilde{F})&\twoheadrightarrow&\widetilde{\Sm}_0
\end{array}
\end{gather}
so that $\psi(x,y)=(x,u(x,y,s))$ with $u\in\mathbb{C}\{x,y\}\Hat{\otimes}\Sh{O}_{\Sm_0}$.
\end{definition}
% TODO replace $u$ by $\Tilde{y}$ and check that this work at all places.

%We use the notion of $x$-morphisms with respect to specified morphisms to $\mathbb{C}_0\hookrightarrow\mathbb{C}_0\twoheadrightarrow\{0\}$ in order to specify the morphisms to the trivial deformation in this definition. We remark that due to \cite[Corollary~II.1.6]{GLS:07} all deformations of $V(f)$ with $f\in\mathbb{C}\{x,y\}$ are indeed of the form $V(f)\hookrightarrow V(F)\twoheadrightarrow\Sm$ with $F\in\mathbb{C}\{x,y\}\Hat{\otimes}\Sh{O}_{\Sm}$. 
We pause briefly to describe what data  explicitly exhibits an $ x $-morphism. Let $f\in\mathbb{C}\{x,y\}$  and $F, \Tilde{F}\in\mathbb{C}\{x,y\}\Hat{\otimes}\Sh{O}_{\Sm_0}$ describe the deformations~\eqref{eq:def germ} as in Definition~\ref{definition:deformation}.  Let $\phi$ denote the map from $\Sm_0 $ to $\widetilde\Sm_0 $ in the $x$-morphism, then the map $V(F)\to V(\Tilde{F})$ is of the form $(x,y,s)\mapsto(x,u(x,y,s),\phi(s))$ with $u\in\mathbb{C}\{x,y\}\Hat{\otimes}\Sh{O}_{\Sm_0}$ and $u(x,y,0)=y $. In particular there exists a unit $H\in\mathbb{C}\{x,y\}\Hat{\otimes}\Sh{O}_{\Sm_0}$ with
\begin{align}\label{eq:mor}
 H(x,y,s)\Tilde{F}(x,u(x,y,s),\phi(s))&=F(x,y,s).
\end{align}
 Evaluating $u$  at $s=0$ yields $u(x,y,0)=y$. If the defining functions are normalised in the sense that $ f (x, y) = F (x, y, 0) = \Tilde{F} (x, y, 0) $, one also has $ H (x, y, 0) = 1 $. % The identity % the map $(x,y)\mapsto(x,u(x,y,0))$ from $V(f)\to V(g)$ with $u(x,y,0)\in\mathbb{C}\{x,y\}$ and the unit $H(x,y,0)\in\mathbb{C}\{x,y\}$. In case we consider deformations of the same $V(f)=V(g)$ we may assume $f=g$ with the identity map i.e.\ $u(x,y,0)=y$. If in addition $F(x,y,0)=f(x,y)$ and $\Tilde{F}(x,y,0)=g(x,y)$, then we also have $H(x,y,0)=1$. 
%
%TODO: Clarify the following paragraph: Phrase in terms of $(X_q, x_q)$? Then we don't need isomorphism classes.

The main result of this section establishes a one-to-one correspondence between isomorphism classes of local $x$-deformations of a locally planar pair $(X,\dfx)$ and $x$-isomorphism classes of deformations of the corresponding $V(f_l)$ at the points $q_l$, $l=1,\ldots,L$ of Definition~\ref{def:locally planar}.
 The additional structure imposed by $\dfx$ on the pair $(X,\dfx)$ serves to restrict the morphisms of deformations of the corresponding germs $ V (f) $ to the $x$-morphisms defined above. 
 
% It is clear that a deformation of $ (X, \dfx) $ induces one of each germ and that the deformation is trivial away from the points $ q_l $ where $ x $ fails to be a local coordinate.  In Theorem ~\ref{equivalent deformation}  we prove moreover that any deformations of the locally planar space germs at $ q_l $ can be patched together with the trivial deformation of $ X\setminus\{q_1,\ldots ,q_L\} $ to yield a (local) deformation of $ (X, \dfx) $ and that the resulting deformation is unique up to $ x $-isomorphism.

% When working with germs $V(f)$ in $\mathbb{C}$, the coordinates are already determined and by Lemma~\ref{le:embbede} a deformation of $V(f)$ includes deformations of both germs $x_q$ and $y_q$. Hence a deformation of $ V (f) $ is automatically an $x$-deformation. 

We now proceed to the more technically difficult existence part of Theorem~\ref {equivalent deformation}. That is, we now show that given deformations of each of the space germs $V(f_l)$ one can patch these together with the trivial deformation to yield an  $ x $-deformation of $(X,\dfx)$.% They are biregular to the zero sets $X_l\simeq V(f_l)$ of $f_l\in\mathbb{C}\{x,y\}$ and $Y_l\simeq V(F_l)$ of $F_l\in\mathbb{C}\{x,y\}\Hat{\otimes}\Sh{O}_{\Sm}$ with $F_l(x,y,0)=H_l(x,y)f_l(x,y)$ with a unit $H_l\in\mathbb{C}\{x,y\}$. We may replace $f_l$ by $H_lf_l$ and obtain $F_l(x,y,0)=f_l(x,y)$. We shall prove now that the isomorphism classes of deformations are essentially determined by the $x$-isomorphism classes of the deformations of space germs $V(f_l)\hookrightarrow V(F_l)\twoheadrightarrow\Sm_0$ with $l$ running through $l\in\{1,\ldots,L\}$.
\begin{theorem}\label{equivalent deformation}
For any locally planar pair $(X,\dfx)$ let $q_1,\ldots,q_L\in X$ and $f_1,\ldots,f_L\in\mathbb{C}\{x,y\}$ be as in Definition~\ref{def:locally planar}. For each $l=1,\ldots,L$ suppose we are given a deformation $V(f_l)\hookrightarrow V(F_l)\twoheadrightarrow\Sm_0$ with $F_l\in\mathbb{C}\{x,y\}\Hat{\otimes}\Sh{O}_{\Sm_0}$. Then there exists a local $x$-deformation $(X,\dfx)\hookrightarrow(Y,\dfx_Y)\twoheadrightarrow\Sm_0$ such that for each $l=1,\dots,L$ the induced deformation $X_l \hookrightarrow Y_l \twoheadrightarrow\Sm_0$ is $x$-isomorphic to the given deformation
$V(f_l) \hookrightarrow V(F_l) \twoheadrightarrow\Sm_0$. This local $x$-deformation is unique up to isomorphism. 
% $V(f_l)\hookrightarrow V(F_l)\twoheadrightarrow\Sm_0$. 
% This condition uniquely determines ~\eqref{eq:deformation 1} up to isomorphism and restriction to a smaller % base.
%
%
%$X_l\hookrightarrow Y_l\twoheadrightarrow\Sm_0$ be a  is $x$-isomorphic to $V(f_l)\hookrightarrow V(F_l) $.
% 
%Let $X$ obey (A1)-(B1), $\Sm_0$ a space germ at $0$ and $F_1,\ldots,F_L\in\mathbb{C}\{x,y\}\Hat{\otimes}\Sh{O}_%{\Sm_0}$ with $F_l(x,y,0)=f_l(x,y)$ for all $l=1,\ldots,L$. Then there exists a complex space $\Sm$ with space %germ $\Sm_0$ at $0\in\Sm$ and a deformation~\eqref{eq:deformation 1}, such that for each $l=1,\dots,L$ the %corresponding $x$-deformation $X_l\hookrightarrow Y_l\twoheadrightarrow\Sm_0$ is $x$-isomorphic to $V(f_l)%\hookrightarrow V(F_l)\twoheadrightarrow\Sm_0$. This condition determines~\eqref{eq:deformation 1} up to %isomorphism and restriction to a smaller base.
\end{theorem}

\begin{proof}
We prove the theorem by proving the following three claims. The uniqueness statement in the theorem then follows from Lemma~\ref{lemma:uniqueness}. 
\begin{itemize}
	\item[\emph{Claim 1.}] $\Sh{S}_0$ can be embedded into $\mathbb{C}^k$ such that there exist an open neighbourhood $\tilde{\Sh{S}}$ of $0$ in $\mathbb{C}^k$, some $\delta_1 > 0$, and for each $l$ Weierstra{\ss} polynomials $\Tilde{f}_l \in \Sh{O}_{B(0,\delta_1)} \Hat{\otimes} \mathbb{C}[y]$ and $\Tilde{F}_l \in \Sh{O}_{B(0,\delta_1)} \Hat{\otimes} \mathbb{C}[y] \Hat{\otimes} \Sh{O}_{\tilde{\Sh{S}}}$ with the following properties: 
	\begin{itemize}
		\item 
		$F_l$ is equal to the germ of $\Tilde{F}_l$ at $0$ multiplied with a unit, and $\Tilde{f}_l$ is the evaluation of $\Tilde{F}_l$ at $0\in \tilde{\Sh{S}}$. 
		\item Defining
		\[V_l=\left\{(x, y,s)\in B(0,\delta_1)\times\mathbb{C}\times\tilde{\Sh{S}}\mid \tilde{F}_l(x,y,s)=0\right\},\]
		the map $V_l \rightarrow B(0,\delta_1)\times\tilde{\Sh{S}},(x,y,s) \mapsto(x,s)$
		is a Weierstra{\ss} covering and is unbranched over $(x,s)\in\Am\times\tilde{\Sh{S}}$ for the annulus $			\Am=B(0,\delta_1)\setminus\overline{B(0,\delta_1/2)}$.
	\end{itemize}		
	\item[\emph{Claim 2.}] After shrinking $\delta_1 $ and $\tilde{\Sh{S}}$ if necessary, we set $C_l:=x_l^{-1} (\overline{B(0,\delta_1/2)})$. Then we may glue the trivial deformation $\big(X\setminus(C_1\cup\ldots\cup C_L)\big)\times\tilde{\Sh{S}}$ with $V_1\cup\ldots\cup V_L$ to give a deformation $(X,\dfx)\hookrightarrow(\tilde{Y},\dfx_{\tilde{Y}})\twoheadrightarrow\tilde{\Sh{S}}$.
	\item[\emph{Claim 3.}] The restriction $(X,\dfx)\hookrightarrow(Y,\dfx_{Y})\twoheadrightarrow{\Sh{S}}_0 \subset \tilde{\Sh{S}}_0$ of the local $x$-deformation corresponding to  $(X,\dfx)\hookrightarrow(\tilde{Y},\dfx_{\tilde{Y}})\twoheadrightarrow\tilde{\Sh{S}}$ is a local $x$-deformation with the property that for every $l$, $X_l\hookrightarrow Y_l\twoheadrightarrow{\Sh{S}}_0$ is $x$-isomorphic to $V(f_l)\hookrightarrow V(F_l)\twoheadrightarrow\Sm_0$. 
\end{itemize}
	
%The uniqueness of ~\eqref{eq:deformation 1} up to isomorphism is an immediate consequence of Lemma~\ref {lemma:uniqueness}. 

\emph{Proof of Claim~1.}
From Definition~\ref{def:locally planar}, each $X_l$ is biregular via $(x_l,y_l)$ to the vanishing set germ $V(f_l)$ of some $f_l\in\C\{x, y\}$. Furthermore, by Lemma~\ref{le:embbede} any deformation of $V(f_l)$ with base $\Sm_0$ is of the form $V(f_l)\hookrightarrow V(F_l)\twoheadrightarrow\Sm_0$ for some $F_l\in\mathbb{C}\{x,y\}\Hat{\otimes}\Sh{O}_{\Sm_0}$, where $F_l(x, y,0)=H_l(x, y)f_l(x, y)$ for a unit $H_l\in\mathbb{C}\{x, y\}$. By replacing $f_l$ by $H_lf_l$ if necessary, we assume without loss of generality that $F_l(x, y,0)=f_l(x, y)$. The space germ $\Sm_0$ is embedded in $\mathbb{C}^k_0$ for some $k$ greater than or equal to the dimension of $\Sm_0$. The holomorphic functions $F_1,\ldots,F_L$ clearly extend to holomorphic functions in $\mathbb{C}\{x,y\}\Hat{\otimes}\Sh{O}_{\mathbb{C}^k_0}$. This extension is not unique, but the restriction to $\Sm_0$ is unique. %We prove the Lemma for $\Sm_0=\mathbb{C}^k_0$ and construct an $ x $-deformation with base equal to an open subset $\Sm$ of $\mathbb{C}^k$. Since the restriction of this $ x $-deformation to a subspace of the base $\Sm$ is an $ x $-deformation whose base is the subspace, the general case follows.

We apply the Weierstra{\ss} Preparation Theorem \cite[Chapter~I Theorem 1.4]{PR:94} and replace $F_1,\ldots,F_L$ by Weierstra{\ss} polynomials $\tilde{F}_1,\ldots,\tilde{F}_L\in\mathbb{C}\{x\}[y]\Hat{\otimes}\Sh{O}_{\mathbb{C}^k_0}$. Furthermore, we replace $f_1,\ldots,f_L$ by the corresponding evaluations $\tilde{f}_1,\ldots,\tilde{f}_L\in\mathbb{C}\{x\}[ y]$ of these Weierstra{\ss} polynomials at $s=0 \in \mathbb{C}^k$. For $l=1,\ldots,L$ the polynomials $\tilde{F}_l$ and $\tilde{f}_l$ have highest coefficient $1$ and all other coefficients vanish at $(x,s)=(0,0)$ and $x=0$, respectively.

%TODO explain strategy of the proof and improve wording

We choose an open ball $\tilde{\Sm}$ around $0\in\mathbb{C}^k$ and $\delta_1>0$ such that the polynomials $\tilde{F}_l$ are holomorphic on $(x, y,s)\in B(0,\delta_1)\times\mathbb{C}\times\tilde{\Sm}$. For sufficiently small $\delta_1$ there exist disjoint open neighbourhoods $O_l$ of $q_l$ in $X$ such that the maps $(x_l, y_l)$ in Definition~\ref{def:locally planar} extend to biregular maps
\begin{align}\label{eq:map b2}
(x_l,y_l)&:O_l\to U_l,&
U_l=\{(x, y)\in B(0,\delta_1)\times\mathbb{C}\mid \tilde{f}_l(x, y)=0\}.
\end{align}
The set $\{q_1,\ldots,q_L\}$ is comprised of the singular points of $X$ and the smooth points which are roots of $\dfx$. Hence 
\begin {align}\label{covering 1}
U_l &\rightarrow B(0,\delta_1),&(x, y) &\mapsto x
\end {align}
is a Weierstra{\ss} covering 
with a single branch point at $(0,0)$. We can take $\delta_1$ sufficiently small and shrink $\tilde{\Sm}$ so that for all $l=1,\ldots,L$,
\[V_l=\left\{(x, y,s)\in B(0,\delta_1)\times\mathbb{C}\times\tilde{\Sm}\mid \tilde{F}_l(x,y,s)=0\right\}\]
is exhibited as a Weierstra{\ss} covering 
\begin{align}\label{covering 2}V_l &\rightarrow B(0,\delta_1)\times\tilde{\Sm},&(x,y,s) &\mapsto(x,s)
\end{align}
and is unbranched over $(x,s)\in\Am\times\tilde{\Sm}$ for the annulus 
\begin{align*}%\label{eq:annulus}
\Am&=B(0,\delta_1)\setminus\overline{B(0,\delta_1/2)}.
\end{align*}

\emph{Proof of Claim~2.}
We denote the restrictions of the coverings~\eqref{covering 1} and~\eqref{covering 2} to $\Am$ and $\Am\times\tilde{\Sm}$ as
\begin{align*}
\widehat{U}_l&\to\Am,& \widehat{V}_l&\to\Am\times\tilde{\Sm},
\end{align*}
respectively. For each $l=1,\ldots,L$ let $C_l:=x_l^{-1} (\overline{B(0,\delta_1/2)})$.
%let $C_l$ denote the preimage in $O_l\subset X $ of $ U\setminus \widehat U_l$ with respect to the map~\eqref{eq:map b2}. 
We shall extend the given deformations of $O_1,\ldots,O_L$ by the trivial deformation
$$X\setminus(C_1\cup\ldots\cup C_L)\hookrightarrow\big(X\setminus(C_1\cup\ldots\cup C_L)\big)\times\tilde{\Sm}\twoheadrightarrow\tilde{\Sm}$$
to $ x $-deformations of the global pair $(X,\dfx)$. For this purpose we glue for any $l=1,\ldots,L$ the open subsets $\big(O_l\setminus C_l\big)\times\tilde{\Sm}$ of $\big(X\setminus(C_1\cup\ldots\cup C_L)\big)\times\tilde{\Sm}$ via the composition of the maps $(x_l,y_l)\times\id_{\tilde{\Sm}}$ with biholomorphic maps $\psi_l:\widehat U_l\times\tilde{\Sm}\to\widehat V_l$ such that the diagram
$$\begin{array}{ccccc}
\big(O_l\setminus C_l\big)\times\tilde{\Sm}&\xrightarrow{(x_l,y_l)\times\id_{\tilde{\Sm}}}&\widehat{U}_l\times\tilde{\Sm}&\xrightarrow{\psi_l}&\widehat{V}_l\\
\\
&&\downarrow&&\downarrow\\
&&\Am\times\tilde{\Sm}&=&\Am\times\tilde{\Sm}.
\end{array}$$
commutes. In accordance with $\tilde{F}_l(x,y,0)=\tilde{f}_l(x,y)$ we assume that the restriction of $\psi_l$ to $\widehat{U}_l\times\{0\}$ is equal to $\id_{\widehat{U}_l\times\{0\}}$.

We now demonstrate for any $l=1,\ldots,L$ the existence of a unique such biholomorphic map $\psi_l$. The coverings $\widehat{U}_l$ and $\widehat{V}_l$ are unbranched Weierstra{\ss} coverings with an equal number of sheets. On $\widehat{U}_l$ therefore, $y $ is locally a holomorphic function $y=y(x) $ of $x\in\Am$. Similarly on $\widehat{V}_l$, locally $ y $ is a holomorphic function $ y = y (x, s) $ with $(x,s)\in\Am\times\tilde{\Sm}$. Furthermore $ y (x, 0) = y (x) $. Each local branch of $y (x) $ extends uniquely to any simply connected subset of $\mathcal A $ containing the original point. Similarly each branch of $ y (x, s) $ extends uniquely on simply connected subsets of $\Am\times\tilde{\Sm}$. Let $W_{l_i} $ be an open subset of $\widehat{U}_l$ which is mapped biholomorphically by $x$ onto an open simply connected subset $\mathcal{B}$ of $\Am$. The cartesian product $W_{l_i}\times\tilde{\Sm}$ is also simply connected since $\tilde{\Sm}$ is an open ball in $\mathbb{C}^k$. Hence there exists a unique biholomorphic map from $W_{l_i}\times\tilde{\Sm}$ onto a simply connected open subset of $\widehat V_l$ which preserves $(x,s)$ and is equal to the identity for $s=0$. 
\[\begin{tikzcd}
W_{l_i}\times\tilde{\Sm} \arrow{r} \arrow[swap]{d}{} &  \widehat{V}_l\arrow{d} \\ \mathcal{B}\times\tilde{\Sm}  \arrow{r} & \mathcal{B}\times \tilde{\Sm}
\end{tikzcd} \quad
\begin{tikzcd}
(x,y (x), s)\rar[maps to]\dar[maps to] & (x, y (x, s), s)\dar[maps to]\\
(x, s)\rar[maps to] & (x, s).
\end{tikzcd}\]
We first cover $\Am$ by finitely many simply connected open subsets $\mathcal{B}$ such that the intersection of any two of them is either empty or connected. Denoting by $ n_l $ the number of sheets of the Weierstrass coverings~\eqref{covering 1} and~\eqref{covering 2}, there are $n_l $ simply connected open subsets $ W_{l_i} $ which each map biholomorphically onto $\mathcal{B} $ via the projection $ x $, and together these cover $\widehat{U}_l$. Since $ \widehat V_l$ is unbranched over $(x,s)\in\Am\times\tilde{\Sm}$, the maps in the above diagram patch together to form a unique biholomorphic map
\begin {align*}
\psi_l:   \widehat U_l\times\tilde{\Sm} &\to  \widehat V_l,& (x, y (x), s) & \mapsto (x, y (x, s), s).
\end {align*}
Writing $\widehat O_l = O_l\setminus C_l $, we have that the open subsets $\widehat O_l\times\tilde{\Sm}$ of $\big(X\setminus(C_1\cup\ldots\cup C_L)\big)\times\tilde{\Sm}$ are mapped by $\psi_l\circ\big((x_l, y_l)\times\id_{\tilde{\Sm}}\big)$ biregularly onto $\widehat{V}_l$. This identification of the two annuli allows us to glue $\big(X\setminus(C_1\cup\ldots\cup C_L)\big)\times\tilde{\Sm}$ with $V_1\cup\ldots\cup V_L$ along the maps $\psi_l\circ\big((x_l, y_l)\times\id_{\tilde{\Sm}}\big)$ for $l=1,\ldots,L$. This defines $\tilde{Y}$ together with the map $\tilde{Y}\twoheadrightarrow\tilde{\Sm}$. The differential $\dfx $ on $X $ extends to a differential on $\big(X\setminus(C_1\cup\ldots\cup C_L)\big)\times\tilde{\Sm}$. Its restriction to $\widehat O_l\times\tilde{\Sm}$  agrees with the pullback under $\psi_l\circ\big((x_l,y_l)\times\id_{\tilde{\Sm}}\big):\widehat O_l\times\tilde{\Sm}\rightarrow \widehat{V}_l $ of the differential $ \dfx$ on $V$. 
Hence we obtain a global differential  $\dfx_{\tilde{Y}} $ on $ \tilde{Y} $ which for $q\notin\{p_1,\ldots , p_K\}$ satisfies $ \dfx_{\tilde{Y}, q} = d (x_{\tilde{Y}, q}) $. The sets $\{p_1,\ldots , p_K\}$ and $\{q_1,\ldots , q_L\} $ are disjoint and so each $p \in\{p_1, \ldots , p_K\} $ has an open neighbourhood $ O_p $ in $ X $ such that $ O_p\times\tilde{\Sm} $ is a neighbourhood of $p $ in $ \tilde{Y} $. Hence the germ  $u_p$ of Definition~\ref{definition:single}~(a)   extends to $u_{\tilde{Y},p}\in\Sh{O}_{\tilde{Y},p}$ such that $(u_{\tilde{Y},p},s)$ maps $\tilde{Y}_p$ biregularly onto $(\mathbb{C}\times\tilde{\Sm})_{(0,0)}$ and $\dfx_{\tilde{Y},p}=d(u_{\tilde{Y},p}^{-M_p})$.
  We claim that the map $\tilde{Y}\twoheadrightarrow\tilde{\Sm}$ is proper and flat. It is the composition of a Weierstra{\ss} map $V_l\twoheadrightarrow B(0,\delta_1)\times\tilde{\Sm}$ and the projection $B(0,\delta_1)\times\tilde{\Sm}\to\tilde{\Sm}$. By \cite[Chapter~II Proposition~2.10]{PR:94} and the Weierstra{\ss} Isomorphism \cite[Chapter~2 \S4.2]{GR:84}, Weierstra{\ss} maps are flat. Furthermore, by \cite[Chapter~II Corollary~2.7]{PR:94} projections of complex spaces are flat. This shows flatness of $\tilde{Y}\twoheadrightarrow\tilde{\Sm}$, since the composition of flat maps is flat \cite[Chapter~II Proposition~2.6]{PR:94}. It is also proper since $X$ and $C_1,\ldots,C_L$ are compact. Hence we have constructed an  $ x $-deformation ~\eqref{eq:deformation 1} with the properties in Claim~2.
  
\emph{Proof of Claim~3.} The restriction of the proper map $\tilde{Y} \twoheadrightarrow\tilde{\Sm}$ to the closed subset $Y$ of $\tilde{Y}$ is again proper. Due to the Weierstra{\ss} Isomorphism \cite[Chapter~2 \S2.]{GR:84}, the direct images of $\Sh{O}_{V_l\cap Y}$ under the restrictions of the Weierstra\ss{} maps $V_l\twoheadrightarrow B(0,\delta_1)\times\tilde{\Sm}$ to $V_l \cap Y$ are locally free. Therefore by the same arguments as in the proof of Claim~2, these restrictions and the map $Y \twoheadrightarrow \Sm$ are flat. 
%
%Furthermore,  flatness of $ \pi: \tilde Y\rightarrow  \tilde \Sm $ means precisely that for each $ y\in  \tilde Y $, the local ring $\mathcal O_{ \tilde Y,y} $ is a flat module over the local ring $\mathcal O_{ \tilde \Sm,  \pi(y)} $. Using the relational characterisation of flatness \cite[Corollary 6.5]{Eisenbud95}, flatness of modules is  preserved under a change of base and so for each $ y\in  Y$, $ \mathcal O_{ \tilde Y,y} \otimes_{ \mathcal O_{ \tilde \Sm, \pi(y)}}\mathcal O_{  {\Sm},\pi(y)} $ is flat over $\mathcal O_{  {\Sm},\pi(y)} $. But $\mathcal O_{   Y, y} =  \mathcal O_{ \tilde Y,y} \hat\otimes_{ \mathcal O_{ \tilde \Sm, \pi(y)}}\mathcal O_{  {\Sm},\pi(y)} $   and this is the same as the regular tensor product, since the local rings are Noetherian and complete with respect to their maximal ideals.
%
%OR find a reference that says that flatness of a map of complex analytic spaces is preserved under a change of base. Stack exchange and Vakil have it for schemes, couldn't find book reference.
This shows that the restriction $(X,dx) \hookrightarrow (Y,dx_Y) \twoheadrightarrow \Sm_0$ of the local $x$-deformation constructed in Claim~2 is again a local $x$-deformation. Because the $\tilde{F}_l$ are extensions of the $F_l$, it follows that for every $l$, $X_l\hookrightarrow Y_l\twoheadrightarrow{\Sh{S}}_0$ is $x$-isomorphic to $V(f_l)\hookrightarrow V(F_l)\twoheadrightarrow\Sm_0$.
\end {proof}

\section{Infinitesimal deformations}\label{sec:infinitesimal}
In this section we will construct a local $ x $-deformation  \eqref {eq:universal 1} of $ (X, \dfx) $ from any basis of the isomorphism classes of infinitesimal $ x $-deformations of $ (X, \dfx) $. %This local deformation is independent of the choice of basis. 
We show that the $x$-deformations constructed from differnt bases are isomorphic.
Furthermore, we prove that the isomorphism classes of infinitesimal $ x $-deformations of the fibres of \eqref {eq:universal representative} form a vector bundle over its base. We will in Section~\ref{sec:existence} use this vector bundle structure to show that the local $ x $-deformations constructed in this section are universal. %The universality will then imply that the dimension of the isomorphism classes of infinitesimal deformations of the fibres of any deformation of $ (X, \dfx) $ are locally constant.
\begin{Definition}\label{def5}
An infinitesimal deformation is a deformation whose base space is the complex space germ $\mathcal{I}_\epsilon$ consisting of a single point with local ring (i.e. holomorphic functions) $\mathcal{I}_\epsilon=\mathbb{C}\{\epsilon\}/\langle\epsilon^2\rangle\simeq\mathbb{C}[\epsilon]/\langle\epsilon^2\rangle$. 
\end{Definition}
Informally, $\mathcal{I}_\epsilon $ may be considered as a point with a single tangent direction. For an analytic curve germ $ V (f) $ in the plane, % ($f\in\mathbb{C}\{x,y\},\, f(0,0)=0$), 
the Tjurina algebra $\mathbb{C}\{x,y\}\big/\langle f,\tfrac{\partial f}{\partial x},\tfrac{\partial f}{\partial y}\rangle $ parameterises isomorphism classes of infinitesimal deformations of $ V (f) $ \cite[Theorem~10.1.5]{dJP:00}. The natural analog holds for $ x $-isomorphism classes: 
\begin{Lemma}\label{l1}
The  $x$-isomorphism classes of the infinitesimal deformations of a space germ $ V (f) $ with $f\in\mathbb{C}\{x,y\},\, f(0,0)=0$ are in one-to-one correspondence with 
\begin{align}\label{eq:quotient}
\mathbb{C}\{x,y\}\big/\langle f,\tfrac{\partial f}{\partial y}\rangle.
\end{align}
\end{Lemma}
\begin{proof}
 An infinitesimal deformation of $ V (f) $ is given by $ V (F) $ for $F(\epsilon)=f+\epsilon g$ and $g\in\mathbb{C}\{x,y\}$. % We shall relax our terminology and refer directly to $ F (\epsilon) $
Suppose we are given two infinitesimal  deformations of $ V (f) $, defined by $F_1(\epsilon)=f+\epsilon g_1$ and $F_2(\epsilon)=f+\epsilon g_2$ respectively. An $ x $-morphism  between them has the form
\[
\begin{array}{cclcl}
V (f) &\hookrightarrow& V (F_1) &\twoheadrightarrow& \mathcal I_\epsilon\\
\|&&\hspace{1mm}\downarrow\psi&&\downarrow\phi\\
V (f) &\hookrightarrow& V (F_2) &\twoheadrightarrow& \mathcal I_\epsilon.
\end{array}
\]
Note however that any automorphism $\phi$ of $\mathcal I_\epsilon $ can be lifted to an automorphism of the infinitesimal deformation. Hence for the purpose of classifying $x$-isomorphism classes of the infinitesimal deformations, we may assume that $\phi$ is the identity. An $x$-morphism for which $\phi $ is the identity is explicitly described, as in ~\eqref{eq:mor}, by  $u,H\in\mathbb{C}\{x,y\}$ such that
$$(1+\epsilon H(x,y))(f(x,y+\epsilon u(x,y))+
\epsilon g_2(x,y+\epsilon u(x,y))=f(x,y)+\epsilon g_1(x,y)$$
holds in $\mathbb{C}\{x,y\}\otimes\Pi$. The $\epsilon^0$-term and the $\epsilon^1$-term yield
\begin{align*}
f(x,y)&=f(x,y),&
H(x,y)f(x,y)+u(x,y)\frac{\partial f}{\partial y}(x,y)+g_2(x,y)&=g_1(x,y).
\end{align*}
We conclude that the two infinitesimal deformations $F_1(\epsilon)=f+\epsilon g_1$ and $F_2(\epsilon)=f+\epsilon g_2$ are $x$-isomorphic if and only if $g_1=g_2$ in $\mathbb{C}\{x,y\}/\langle f,\frac{\partial f}{\partial y}\rangle$.
\end{proof}

Our next task is to show, in Theorem~\ref{l3}~(ii), that the spaces of isomorphism classes of infinitesimal $ x $-deformations of the fibres of any $x$-deformation form a vector bundle over the base. The rank of the vector bundle is the degree of a sheaf of regular differentials with certain poles and hence is given in terms of the arithmetic genus. In the proof we shall use the following local description of the regular forms (Rosenlicht differentials) away from the poles of $\dfx$:
\begin{Lemma}\label{l2}
Let $f\in\mathbb{C}[y]\Hat{\otimes}\Sh{O}_W$ be a Weierstra{\ss} polynomial of degree $d$ with respect to $y$, whose coefficients are holomorphic functions on the open subset $x\in W\subset\mathbb{C}$. If the roots of the discriminant of $f$ are isolated in $W$, then the regular forms of the complex space $O=\{(x,y)\in W\times\mathbb{C}\mid f(x,y)=0\}$ have the following form:
$$\omega=g(x,y)\dfx\big/\frac{\partial f}{\partial y}(x,y),\,
\text{with }g\in\Sh{O}_O.$$
\end{Lemma}
\begin{proof}
By definition, a regular form is a meromorphic form whose product with any local holomorphic function $h$ has no residues \cite[Chapter~IV \S3]{Serre:88} and \cite{Rosenlicht:52}. The statement of the lemma is a local statement. Therefore it suffices to show that for each $(x_0,y_0)\in O$ there exists a ball $B(x_0,\epsilon)\subset W$ such that the lemma holds on the restriction of the Weierstra{\ss} covering $O\to W$ to the preimage $\Tilde{O}$ of $B(x_0,\epsilon)$. In particular, we may choose $B(x_0,\epsilon)$ small enough such that all the connected components of $\Tilde{O}$ contain only one point $(x_0,y_0)$ over $x_0$. Since the holomorphic functions on different connected components are independent, it suffices to prove the statement of the lemma for each connected component of $\Tilde{O}$ separately. Consequently we may assume that $\Tilde{O}\to B(x_0,\epsilon)$ is a Weierstra{\ss} covering with only one point $(x_0,y_0)$ over $x_0$. Then for a meromorphic 1-form $\omega$ on $\Tilde{O}$, the residue of $h\omega$ at $(x_0,y_0)$ is the residue at $x_0$ of the meromorphic 1-form on $B(x_0,\epsilon)$ given by the sum of $h\omega$ over the sheets of $\Tilde{O}\to B(x_0,\epsilon)$.

Since this sum of $h\omega$ is a meromorphic 1-form on $B(x_0,\epsilon)$, a meromorphic 1-form $\omega$ is regular at $(x_0,y_0)$ if and only if for all $h\in\Sh{O}_{\Tilde{O}}$, the sum of $h\omega$ over the sheets is holomorphic at $x_0$. Furthermore, due to the Weierstra{\ss} Division Theorem~\cite[Chapter I Theorem 1.1]{PR:94} it suffices to satisfy this condition for all $h=y^n$ with $n=0,\ldots,d-1$ where $d$ denotes the number of sheets.
%Therefore the meromorphic 1-form $\omega$ is regular at $(x_0,y_0)$ if and only if for all $h\in\Sh{O}_{\Tilde{O}}$, the sum of $h\omega$ over the $d$ sheets of $\Tilde{O}\to B(x_0,\epsilon)$ has no residue at $x_0$. This implies that this sum of $h\omega$ is holomorphic at $x_0$ as a 1-form on $x\in B(x_0,\epsilon)$. Hence it suffices to prove that a meromorphic function $g=\frac{\partial f}{\partial y}\frac{\omega}{\dfx}$ on $\Tilde{O}$ is holomorphic at $(x_0,y_0)$, if and only if for all $h\in\Sh{O}_{\Tilde{O}}$ the sum of $hg/\frac{\partial f}{\partial y}$ over the $d$ sheets of $\Tilde{O}\to B(x_0,\epsilon)$ is holomorphic at $x_0$.

To proceed we fix $x\in B(x_0,\epsilon)$ and derive a formula for such sums in the case of a complex polynomial $\Tilde{f}(y):=f(x,y)$ of degree $d$ with highest coefficient $1$ and pairwise different roots $y_1,\ldots, y_d$. For any polynomial $\Tilde{g}(y)$ of degree less than $d$ we have the identity
\begin{equation}\label{eq:pol 1}
\Tilde{g}(y)=\sum_{i=1}^d\Tilde{g}(y_i)\prod_{j\ne i}(y-y_j)\big/\frac{\partial\Tilde{f}_0}{\partial y}(y_i).
\end{equation}
To see this, observe that the values of both sides at $y=y_1,\ldots, y_d$ coincide. Since both sides are polynomials with respect to $y$ of degree less than $d$ they are equal. In particular, the sum $\sum_{i=1}^d\Tilde{g}(y_i)/\frac{\partial\Tilde{f}}{\partial y}(y_i)$ is equal to the coefficient of the monomial $y^{d-1}$ in $\Tilde{g}(y)$. For the monomials $\Tilde{g}(y)=1,y,\ldots,y^{d-1}$ this implies
\begin{equation}\label{eq:pol 2}
\sum_{i=1}^dy_i^n\big/\frac{\partial\Tilde{f}}{\partial y}(y_i)=\left\{
\begin{array}{ll}
0&\mbox{ if }0\le n<d-1\\
1&\mbox{ if }n=d-1.\end{array}\right.
\end{equation}
The same is true for $f$, for which the coefficients are holomorphic functions depending on $x\in B(x_0,\epsilon)$. %Due to the Weierstra{\ss} Divison Theorem \cite[Chapter I Theorem 1.1]{PR:94} each holomorphic function on $\Tilde{O}$ may be written as a polynomial with respect to $y$ of degree less than $d$ whose coefficients are holomorphic functions depending on $x\in B(x_0,\epsilon)$. Furthermore, the meromorphic functions on $\Tilde{O}$ can be written as polynomials with respect to $y$ of degree less than $d$, whose coefficients are meromorphic functions on $x\in B(x_0,\epsilon)$.
We multiply any meromorphic function $g$ on $\Tilde{O}$ with the monomials $y^0,y^1,\ldots, y^{d-1}$ and subtract  from the product multiples of $f$ such that the difference again has degree less than $d$ with respect to $y$. Consequently, a meromorphic function $g$ on $\Tilde{O}$ is holomorphic if and only if for $n=0,\ldots, d-1$, the sum of $y^n gdx/\frac{\partial f}{\partial y}$ over the sheets of $\Tilde{O}\to B(x_0,\epsilon)$ is holomorphic on $B(x_0,\epsilon)$. Now setting $\omega= gdx/\frac{\partial f}{\partial y}$ finishes the proof.
%Therefore $\sum_{i=1}^dg(y_i)y_i^n/\frac{\partial f}{\partial y_i}(y_i)$ is holomorphic at $x_0$ for holomorphic $g$ and $n=1,\ldots,d-1$. Conversely, we may conclude inductively that if $\sum_{i=1}^dg(y_i)y_i^n/\frac{\partial f}{\partial y_i}(y_i)$ are holomorphic at $x_0$ for $n=0,\ldots,d-1$, then all coefficients of $g$ are holomorphic at $x_0$.
\end{proof}
\begin{theorem}\label{l3}
Let $(X,\dfx)$ be a locally planar pair with prescribed poles of orders $M_1,\ldots,M_K$ at the marked points $p_1,\ldots,p_K$, respectively. Let $D=\sum_k(M_k+1)\cdot p_k$ be the pole divisor of $\dfx$ and $\Omega_X^1(D)$ the sheaf of meromorphic differentials on $X$, which are regular on $X^\circ=X\setminus\{p_1,\ldots,p_K\}$ and have poles of order at most $M_k+1$ at $p_k$ for any $k=1,\ldots,K$. Then
\begin{enumerate}
\item[(i)] The space of isomorphism classes of infinitesimal $x$-deformations of $(X,\dfx)$ is isomorphic to $H^0(X,\Omega^1_X(D)/(\Sh{O}_Xdx))$.
\item[(ii)] For any $x$-deformation $(X,\dfx)\hookrightarrow(Y,\dfx_Y)\xtwoheadrightarrow{\pi}\Sm$ of $(X,\dfx)$, the spaces of isomorphism classes of infinitesimal $x$-deformations of the fibres of $\pi$ comprise a vector bundle over $\Sm$ of rank equal to
$$\deg\Omega^1_X(D)=2g-2+\sum_{k=1}^K(M_k+1).$$
\end{enumerate}
%Let $ g_l,\, l = 1,\ldots, L $ be a basis for the isomorphism classes of infinitesimal deformations of locally planar $ (X, \dfx) $ and
%\[(X,\dfx)\hookrightarrow(Z,\dfx_Z)\twoheadrightarrow\Tm\]
%be the induced deformation of $ (X, \dfx) $, as in  \eqref {eq:universal representative}. The  isomorphism classes of infinitesimal deformations of  the fibres of  \eqref {eq:universal representative} comprise a  vector bundle over $\Tm$. On a neighbourhood of $ 0\in\Tm $, the  $ g_l $ are  linearly independent sections and hence provide a trivialisation of this bundle.
\end{theorem}
\begin{proof}
Due to Theorem~\ref{equivalent deformation} the space of isomorphism classes of infinitesimal $x$-deformations of the pair $(X,\dfx)$ is the direct sum of the spaces of infinitesimal $x$-isomorphism classes of $V(f_1),\dots,V(f_L)$. By Lemma~\ref{l1} this is equal to
$$\bigoplus_{l=1}^L\mathbb{C}\{x_l,y_l\}/\langle f_l,\tfrac{\partial f}{\partial y_l}\rangle=\bigoplus_{l=1}^L\big(\mathbb{C}\{x_l,y_l\}/\langle f_l\rangle\big)/\langle\tfrac{\partial f_l}{\partial y_l}\rangle=\bigoplus_{l=1}^L\Sh{O}_{X,q_l}/\tfrac{\partial f_l}{\partial y_l}\Sh{O}_{X,q_l}.$$
Here $q_l$, $x_l$, $y_l$ and $f_l$ are as in Definition~\ref{def:locally planar}. By viewing both stalks $\Sh{O}_{X,q_l}$ and $\tfrac{\partial f_l}{\partial y_l}\Sh{O}_{X,q_l}$ as subsets of the stalks of the meromorphic functions, we may divide them by the meromorphic function $\frac{\partial f_l}{\partial y_l}$. The foregoing lemma gives $\Sh{O}_{X,q_l}\big(\tfrac{\partial f_l}{\partial y_l}\big)^{-1}\big/\Sh{O}_{X,q_l}\simeq\Omega^1_{X,q_l}(D)/(\Sh{O}_{X,q_l}dx_l)$. Since $\{q_1,\ldots,q_L\}$ is the support of $\Omega^1_X(D)/\Sh{O}_Xdx$, part~(i) follows. The choice of the local germs $y_l $ is not unique, as we may postcompose $(x_l,y_l)$ with any biregular automorphism of $\mathbb{C}^2$ that preserves the first coordinate. For given $y_l$ the germ $f_l$ has a further ambiguity of multiplication by a unit. However the resulting spaces $\mathbb{C}\{x,y\}\big/\langle f_l,\tfrac{\partial f_l}{\partial y_l}\rangle $ of $ x $-isomorphism classes of infinitesimal deformations (Lemma~\ref{l1}) are isomorphic.

The statement in part~(ii) is local on $\Sm$. Hence we may assume as in the proof of Theorem~\ref{equivalent deformation} that $\Sm$ is smooth and $Y_l$ is for any $l=1,\ldots,L$ given by the zero set of $F_l\in\mathbb{C}\{x_l,y_l\}\Hat{\otimes}\Sh{O}_{\Sm}$. Let $Z_l\subset Y_l$ be the zero set of $\frac{\partial F_l}{\partial y_l}$. We claim that for sufficiently small $\Sm$ the following map is flat:
$$\pi':Z=Z_1\cup\ldots\cup Z_L\hookrightarrow Y\xtwoheadrightarrow{\pi}\Sm.$$
In fact, for sufficiently small $\Sm$ the fibres of this map are the union of the singularities of the corresponding fibres of $\pi$ together with the smooth roots of the restriction of $\dfx_Y$ to the corresponding fibres of $\pi$. Hence the map $\pi'$ is finite. As proven in \cite[Chapter II Proposition 2.10]{PR:94}, a finite holomorphic map $\pi':Z\twoheadrightarrow\Sm$ to a reduced complex space is flat precisely when the sum of the dimensions of the quotients of $\Sh{O}_Z$ divided by the functions which vanish on the fibres $(\pi')^{-1}[\{s\}]$ of $\pi'$ is locally constant on $s\in\Sm$. Due to part~(i), at $s\in\Sm$ this dimension is equal to
\begin{gather}\label{deg dim}
\deg\Omega_{\pi^{-1}[\{s\}]}(D)=2g-2+\deg(D)=2g-2+\sum_{k=1}^K(M_k+1).
\end{gather}
Here $g$ is the arithmetic genus of $\pi^{-1}[\{s\}]$. Due to the Semi-Continuity Theorem~\cite[Chapter III Theorem 4.7c]{PR:94} this arithmetic genus is constant on $s\in\Sm$ and hence $\pi'$ is flat. Furthermore $\pi'$ is proper, since $\pi$ is proper and $Z$ is closed in $Y$ The Semi-Continuity Theorem also states \cite[Chapter III Theorem 4.7d]{PR:94} that the direct image of the structure sheaf under a proper flat map is a vector bundle. Applying this result here yields that the direct image of the structure sheaf of $Z$ with respect to $\pi'$ is the sheaf of holomorphic sections of a holomorphic vector bundle on $\Sm$ of rank~\eqref{deg dim}. Due to part~(i) and by definition of $Z$, this is the statement that the isomorphism classes of infinitesimal $x$-deformations form a vector bundle on $\Sm$.
\end{proof}
For each $ l = 1,\ldots , L $ we utilised above a choice of $ y_l\in\Sh{O}_{X_l, q_l}$ such that $y_l (q_l) = 0  $ and $ (x_l, y_l): X_l\rightarrow V(f_l) $ is biregular for some $ f_l\in\mathbb C\{x, y\} $. This $y_l $ is not unique: for any biregular $v:\C\rightarrow\C $, we may instead take $v\circ y_l $. The corresponding spaces $\mathbb{C}\{x,y\}\big/\langle f_l,\tfrac{\partial f_l}{\partial y_l}\rangle $ of $ x $-isomorphism classes of infinitesimal deformations (Lemma~\ref{l1}) are  isomorphic.

\begin{lemma}
	\label{prep-for-univ-deformation-construction}
	%Let $(X,\dfx) $ be locally planar and let $q_l$, $x_l$, $y_l$ and $f_l$ be as in Definition~\ref{def:locally planar} for $l=1,\ldots,L$. 
	Take $\delta_1>0$ and a Weierstraß polynomial $f \in \mathbb{C}[y]\Hat{\otimes}\Sh{O}_{B(0,\delta_1)}$ such that
	\begin{itemize}
		\item $f$ has highest coefficient $1$ and lower coefficients vanishing at $x=0$,
		\item $q=(0,0)$ is the unique point in $B(0,\delta_1)\times \mathbb{C}$ so that $f(q) = \tfrac{\partial f}{\partial y}(q)=0$.
	\end{itemize}
	Then
	$\mathbb{C}\{x,y\}/\langle f,\tfrac{\partial f}{\partial y}\rangle \simeq \mathbb{C}\{x\}[y]/\langle f,\tfrac{\partial f}{\partial y}\rangle\simeq\mathbb{C}[y]\Hat{\otimes}\Sh{O}_{B(0,\delta_1)}/\langle f,\tfrac{\partial f}{\partial y}\rangle$ and these spaces are finite-dimensional.
\end{lemma}

\begin{proof}
Applying the Weierstra{\ss} Divison Theorem \cite[Chapter I Theorem 1.1]{PR:94}, we have that the quotient $\mathbb{C}\{x,y\}/\langle f\rangle$ is the same as $\mathbb{C}\{x\}[y]/\langle f\rangle$. This implies $\mathbb{C}\{x,y\}/\langle f,\frac{\partial f}{\partial y}\rangle\simeq\mathbb{C}\{x\}[y]/\langle f,\frac{\partial f}{\partial y}\rangle$. 

The complex space $U=V(f)$ is a Weierstra{\ss} covering over $x\in B(0,\delta_1)$. We have $\mathbb{C}[y]\Hat{\otimes}\Sh{O}_{B(0,\delta_1)}/\langle f\rangle \simeq H^0(U,\Sh{O}_U)$ and therefore it follows that also $\mathbb{C}[y]\Hat{\otimes}\Sh{O}_{B(0,\delta_1)}/\langle f,\tfrac{\partial f}{\partial y}\rangle \simeq H^0\bigr(U,\Sh{O}_U/(\tfrac{\partial f}{\partial y } \Sh{O}_U)\bigr)$. By the second assumption, the sheaf $\Sh{O}_U/(\tfrac{\partial f}{\partial y } \Sh{O}_U)$ is a skyscraper sheaf with support at $(0,0)$. This implies the second isomorphism in the lemma and that the quotients are finite-dimensional.
\end{proof}

We now explain how any basis for the isomorphism classes of infinitesimal deformations of a locally planar $(X,\dfx) $ induces a local $ x $-deformation of $ (X,\dfx)$, which we later prove to be universal. 
Let $(X,\dfx) $ be locally planar, and let $q_l$, $x_l$, $y_l$ and $f_l$ be as in Definition~\ref{def:locally planar} for $l=1,\ldots,L$.
 By the Weierstra{\ss} Preparation Theorem \cite[Chapter I Theorem 1.4]{PR:94} we may replace the $f_l$ by Weierstra{\ss} polynomials in $\mathbb{C}[y]\Hat{\otimes}\Sh{O}_{B(0,\delta_1)}$ of degree $d_l$ in $y$, for some $\delta_1>0$.
The common roots of $f_l$ and $\tfrac{\partial f_l}{\partial y}$ are either singularities of $X$ or smooth zeros of $\dfx$. Therefore the conditions of Lemma~\ref{prep-for-univ-deformation-construction} are satisfied if we choose $\delta_1$ small enough so that the balls $B(0,\delta_1)$ do not contain more than one of the points $q_l$. 

By Lemmata~\ref{l1} and \ref{prep-for-univ-deformation-construction}, the isomorphism classes of infinitesimal deformations of $(X,\dfx) $ are given by  $ \bigoplus _{l = 1} ^ L \mathbb{C}[y]\Hat{\otimes}\Sh{O}_{B(0,\delta_1)}/\langle f_l,\tfrac{\partial f_l}{\partial y}\rangle$ and this space is finite-dimensional.  For each $l=1,\ldots,L$ then, we choose a basis of $\mathbb{C}[y]\Hat{\otimes}\Sh{O}_{B(0,\delta_1)}/\langle f_l,\tfrac{\partial f_l}{\partial y}\rangle$, represented by tuples 
\begin {equation}\label {equation:basis}
 g_l=(g_{l,1},\ldots, g_{l,r_l})\in(\mathbb{C}[y]\Hat{\otimes}\Sh{O}_{B(0,\delta_1)})^{r_l}.
\end {equation}
 By the Weierstra{\ss} Division Theorem, we may choose the degrees of the $g_{l,r_i}$ to be at most $d_l-1$. We further choose small open balls $\Tm_l\subset\mathbb{C}^{r_l}$ such that for all $t_l\in\Tm_l$, the roots of the discriminant with respect to $y$ of the following polynomials $G_l$
\begin{align} % U_l(t_l)&=\{(x,y)\in B(0,\delta_1)\times\mathbb{C}\mid G_l(x,y,t_l)=0\}
	%\quad\mbox{with}\\\nonumber
	G_l(x,y,t_l)&=f_l(x,y)+t_{l,1}g_{l,1}(x,y)+\ldots+t_{l,r_l}\label {eq:G}
	g_{l,r_l}(x,y)\\&=:f_l(x,y)+t_l\cdot g_l(x,y)\nonumber
\end{align}
belong to $x\in B(0,\delta_1/2)$. Thereby via  \eqref {eq:G}
each $ g_l $ defines a deformation
\begin {equation}\label{eq:deformation}
U_l\rightarrow V_l \twoheadrightarrow \mathcal T =\mathcal T_1\times\cdots\mathcal \times \mathcal{T}_L 
\end {equation}
where
\begin{align}\label{eq:zl}
	V_l&=\{(x,y,t_1,\ldots,t_L)\in B(0,\delta_1)\times\mathbb{C}\times\Tm\mid G_l(x,y,t_l)=0\}.
\end{align}
% As in the proof of Theorem~\ref{equivalent deformation}, % &we choose small open balls $\Tm_l\subset\mathbb{C}^r$, such that the roots of the discriminant of the polynomial $f_l+t_l\cdot g_l$ with respect to $y_l$ belong for all $t_l\in\Tm_l$ to $x\in B(0,\delta_1/2)$. Again 
Each $ V_l $ is a Weierstra{\ss} covering with $ d_l $ sheets over $x\in B(0,\delta_1)\times\Tm$, unbranched over $\Am\times\Tm$ with $\Am=B(0,\delta_1)\setminus\overline{B(0,\delta_1/2)}$. These deformations satisfy the conclusion of Claim~1 of Theorem~\ref{equivalent deformation}, and hence Claim~2 of Theorem~\ref{equivalent deformation} is applicable.
% the annulus in~\eqref{eq:annulus}. This satisfies Claim 1 of Theorem~\ref{equivalent deformation}. 
% glue $\big(X\setminus(C_1\cup\ldots\cup C_L)\big)\times\Tm$ with $V_1\cup\ldots\cup V_L$ in such a way  that $d(x_1),\ldots,d(x_L)$ extend to a global meromorphic 1-form obeying the conditions (a)-(b) of Definition~\ref {def2}. (TODO: Rework Theorem 2.6 to eliminate reference to proof?)

\begin {definition}\label{universal1} % \label{eq:universal 1}
By Claim 2 of Theorem~\ref{equivalent deformation}, we may glue the deformations~\eqref{eq:deformation} with the trivial deformation to obtain an $x$-deformation 
\begin {equation} (X,\dfx)\hookrightarrow(Z,\dfx_Z)\twoheadrightarrow\Tm \; .  \label{eq:universal representative}
\end {equation}
We call this the $ x $-deformation  of $ (X, \dfx) $ defined by the choice of basis $ g_l,\, l = 1,\ldots , L $ \eqref{equation:basis} for the isomorphism classes of infinitesimal $ x $-deformations of  $ (X, \dfx) $. The corresponding local $ x $-deformation  of $ (X, \dfx) $ is denoted by
\begin{equation} \label{eq:universal 1}
	(X,\dfx)\hookrightarrow(Z,\dfx_Z)\twoheadrightarrow\Tm_0.
\end{equation}
\end{definition}
We shall show in Theorem~\ref{t4} that this local $ x $-deformation is universal.

If in addition $(X,dx)$ is real with respect to an anti-holomorphic involution $\eta$ (Definition~\ref{definition:single}(c)), then we can choose the basis $g_1,\dotsc,g_L$ so that the $x$-deformation \eqref{eq:universal representative} is also real with respect to $\eta$ (Definition~\ref{def2}(c)). Indeed, the maps $(x_l,y_l):X_l\to\mathbb{C}^2$ obey \eqref{eq:real}. We define an action $l\mapsto\eta l$ of $\eta$ on $\{1,\ldots,L\}$ such that $\eta(q_l)=q_{\eta l}$. Due to \eqref{eq:real} the unique Weierstra{\ss} polynomials $f_1,\ldots,f_L$ obey
$$f_{\eta l}(x,y)=\Bar{f}_l(\Bar{x},\Bar{y}).$$
We choose the basis $g_1,\ldots,g_L$ so that:
\begin{equation}\label{eq:real 2}
g_{\eta l}(x,y)=\Bar{g}_l(\Bar{x},\Bar{y}).
\end{equation}
In particular, for $\eta l=l$ the coefficients of $g_l$ take real values for real $x_l\in B(0,\delta_1)$. Consequently the deformation~\eqref{eq:universal representative} obeys condition~(c) in Definition~\ref{def2} with $\eta(t_l)=\Bar{t}_{\eta l}$ also denoted by $\eta$.
%%%%%%%%%%%%%%%%%%%%%%%%%%%%%%%%%%%%%%%%%%%%%%%%%%%%%%%%%%%%%%%%%%%%%
\section{Decomposition of holomorphic functions}\label {sec:decomposition}
 In this section we prove a decomposition result which we shall use in our proof that for each choice of basis $g_1,\ldots,g_L$ for the isomorphism classes of infinitesimal deformations of $ (X,\dfx) $ as in  \eqref {equation:basis}, the local $ x $-deformation \eqref{eq:universal 1} is universal. Universality means that given a local $ x $-deformation $(X,\dfx)\hookrightarrow(Y, \dfx_Y)\twoheadrightarrow\mathcal S_0 $,  there exists a unique holomorphic map $\phi:\mathcal S_0\rightarrow\mathcal T_0 $ such that the pullback of  \eqref {eq:universal 1} under $\phi $ is isomorphic to the given local $ x $-deformation.  

As explained in Theorem~\ref{equivalent deformation}, such an isomorphism is determined by $x$-isomorphisms
$$\begin{array}{ccccl}
V(f_l)&\hookrightarrow& V(F_l)&\twoheadrightarrow&\Sm_0\\
\| &&\downarrow&&\downarrow\\
V(f_l)&\hookrightarrow&V(G_l)&\twoheadrightarrow&\Tm_0
\end{array}$$
of deformations of the corresponding germs $ X_l = X_{q_l} $ for $l=1,\ldots, L$. Explicitly, for each $l=1,\ldots, L$ these $x$-morphisms are comprised  of a triple $(\phi_l,u_l,H_l)$, as in \eqref{eq:mor}, such that
\begin {equation}\label {equation:morphism}
H_l (x,y,s)G_l (x,u_l (x,y,s),\phi_l(s))=F_l (x,y,s),
\end {equation}
where $ G_l $ was defined in~\eqref {eq:G}. 
%and $ (x_{Y_l}, y_{Y_l}) $ defines a biregular map from $ Y_l $ to $ V (F_l) $.
We shall differentiate~\eqref{equation:morphism} with respect to $s$ in order to derive a vector field for triples $(\phi_l,u_l,H_l)$ and then apply the Picard-Lindel\"of Theorem. 
In the next lemma we prove that the derivative of the right hand side in~\eqref{equation:morphism} with respect to $s$ can be written uniquely in the form $a\cdot G_l + b \cdot \frac {\partial G_l} {\partial u_l} + c\cdot g_l$, where the properties of $a,b,c$ will be specified below. Taking the derivative of the left hand side of~\eqref{equation:morphism}, this decomposition allows us to express the derivatives of $\phi_l$, $u_l$ and $H_l$ in terms of $a,b,c$.

% In order to prove that Definition~\ref {universal1} is a universal deformation we need some additional preparation.

We supplement the choice of $\delta_1>0$ in the proof of Theorem ~\ref{equivalent deformation} by a $\delta_2>0$ chosen so that for each $l=1,\ldots,L$, the complex space $U_l$ in~\eqref{eq:map b2} is a subset of the multidisc $(x,y)\in B(0,\delta_1)\times B(0,\delta_2)$. Furthermore, we choose the balls $\Tm_l\subset\mathbb{C}^{r_l}$ sufficiently small such that $V_l$~\eqref{eq:zl} is a subset of the multidisc $(x,y,t_l)\in B(0,\delta_1)\times B(0,\delta_2)\times\Tm_l$. Then % For positive $\delta=(\delta_1,\delta_2)$ 
let $\mathfrak {H}$ denote the Banach space of bounded holomorphic functions on $(x,y)\in B(0,\delta_1)\times B(0,\delta_2)$ with the supremum norm $\|\cdot\|_\infty$. We define
$$\mathfrak {H}_l:=\big\{u\in\Sh{O}_{B(0,\delta_1)}\Hat{\otimes}\mathbb{C}[y]\cap\mathfrak{H}\mid\deg u<d_l\big\},$$ %to be the subset of $ \mathfrak {H} $ consisting of those elements which are polynomials in $ y $ of degree less than $ d_l $.
where $d_l$ is the degree of the Weierstra{\ss} covering
$ x: V_l \to B (0,\delta_1)\times\mathcal T_l$.

\begin{Lemma}\label{l5}
% Take  $l\in\{1,\ldots,L\}$.
Let $g_l$ be given as in~\eqref {equation:basis} and let $G_l, \delta_1,\delta_2,\mathcal{T}_l,\|\cdot\|_\infty,\mathfrak{H},\mathfrak{H}_l$ be as above. There exist $\epsilon_1,\epsilon_2 > 0 $ such that for any $(u,t_l)\in\mathfrak{H}_l\times\Tm_l$ with $\|u-y\|_\infty<\epsilon_2 $ and $ |t_l| <\epsilon_1 $, every $h\in\mathfrak{H}$ has a unique decomposition as
\begin{gather}\label{f3}\begin{aligned}
h (x, y)=&\\a(x,y)G_l&(x,u(x,y),t_l)+b(x,y)\frac{\partial G_l}{\partial u}(x,u(x,y),t_l)+c\cdot g_l(x,u(x,y)),
\end{aligned}\end{gather}
where $ (a,b,c)\in\mathfrak {H}\times \mathfrak {H}_l\times \mathbb{C}^{r_l} $.
Furthermore $(a,b,c)$ depends holomorphically on such $(h,u,t_l)\in\mathfrak {H}\times\mathfrak {H}_l\times\Tm_l$.
\end{Lemma}
\begin{proof}
In the Lemma we fix $l\in\{1,\ldots,L\}$ and hence we omit the index $l$ of $u_l,\, h_l$ and $F_l$ appearing in~\eqref{equation:morphism}.
Since $u$ is a polynomial with respect to $y$, for small $\|u-y\|_\infty$ the derivative $\frac{\partial u}{\partial y}-1$ is small. This means that for $(x,y)\in B(0,\delta_1)\times B(0,\delta_2)$, the maps $v\mapsto u(x,v)$ and $v\mapsto y-(u(x,v)-v)$ are contractions on $B(0,\delta_2)$. Therefore, by Banach's Fixed Point Theorem, for all $(x,y)\in B(0,\delta_1)\times B(0,\delta_2)$ the map $v\mapsto y-(u(x,v)-v)$ has a unique fixed point $v(x,y)$ on $ B(0,\delta_2)$, obtained by iterating this map with the initial input $v=0$. Hence the map $(x,y)\mapsto (x,u(x,y))$ is a biholomorphism from an open subset $U$ of $\mathbb{C}^2$ onto $B(0,\delta_1)\times B(0,\delta_2)$, with inverse map $\varphi: (x,y)\mapsto(x,v(x,y))$. The function $\Tilde{h}(x,y)=h(x,v(x,y))$ is holomorphic on $B(0,\delta_1)\times B(0,\delta_2)$ and bounded.
For $ t_l\in\mathcal T_l$ we set
\begin{align*}
U_l(t_l) &:= \big\{ (x,y) \in B(0,\delta_1)\times B(0,\delta_2) \mid G_l(x,y,t_l) = 0 \big\}.%\\
%&=\big\{(x,y)\in B(0,\delta_1)\times B(0,\delta_2)\mid G_l(x,y,t_l)=0\big\}.
\end{align*}
Due to Theorem~\ref{l3}~(ii), for sufficiently small $ t_l\in\mathcal T_l$,  the $g_l$ form a basis of  $\Sh{O}_{U_l(t_l)}\big/\frac{\partial G_l}{\partial y}(x,y,t_l)\Sh{O}_{U_l(t_l)}$.
%\eqref {eq:basis g}. 
Hence there exists a unique $c\in\mathbb{C}^{{r_l}}$ such that
\begin{align}\label{eq:unique c}
\left(\Tilde{h}(x,y)-c\cdot g_l(x,y)\right)\big/\frac{\partial G_l}{\partial y}(x,y,t_l)
\end{align}
is holomorphic on the complex space $U_l(t_l)$. %Here we assume that $\delta_1>0$ is for our choice of $\delta_2>0$ sufficiently small that $U_l(t_l)$ is contained in $U$.
The space $U_l(t_l)$ is by definition a $ d_l $-fold Weierstra{\ss} covering over $x\in B(0,\delta_1)$. For each fixed $x$, the map $y\mapsto\varphi(x,y)$ takes all $d_l$ roots of $G_l(x,y,t_l)$ into $\{v\in\mathbb{C}\mid (x,v)\in U\}$. Hence the image $\varphi[U_l(t_l)] \subset U$ is also a $ d_l $-fold Weierstra{\ss} covering over $x\in B(0,\delta_1)$. For sufficiently small $|t_l|$ and $\|u-y\|_\infty$, the polynomial $y\mapsto G_l(x,u(x,y),t_l)$ has exactly $ d_l $ roots in $y\in B(0,\delta_2)$, or equivalently
\begin{align}\label{eq U subset ball}
\varphi[U_l(t_l)]&\subset B(0,\delta_1)\times B(0,\delta_2).
\end{align}
In particular, there exists a unique Weierstra{\ss} polynomial $F\in\mathfrak{H}$ such that
\[ \varphi\big[U_l(t_l)\big]  =\big\{(x,y)\in B(0,\delta_1)\times B(0,\delta_2)\mid F(x,y)=0\big\}.\]
 We conclude that the function
\begin{equation}\label{eq:b}
\left(h(x,y)-c\cdot g_l(x,u(x,y))\right)\big/\frac{\partial G_l}{\partial u}(x,u(x,y),t_l)
\end{equation}
is holomorphic on $ \varphi[U_l (t_l)] $. By the Weierstra{\ss} Division Theorem \cite[Chapter I Theorem 1.1]{PR:94}, this function is equal to a unique $b\in\mathfrak {H}_l $ on the zero set of $F$. Hence
$$h(x,y)-c\cdot g_l(x,u(x,y))-b(x,y)\frac{\partial G_l}{\partial u}(x,u(x,y),t_l)$$
vanishes on 
% $$\{(x,y)\in B(0,\delta_1)\times B(0,\delta_2)\mid F(x,y)=0\}$$
$$\varphi\big[U_l (t_l)\big]=\big\{(x,y)\in B(0,\delta_1)\times B(0,\delta_2)\mid G_l(x,u(x,y),t_l)=0\big\}.$$
Since $ \varphi ^ {- 1} $ is a biholomorphism from this space
% $(x,y)\mapsto(x,u(x,y))$ maps this space biholomorphically onto 
to the reduced space $ U_l (t_l) $, % $\{(x,y)\in U\mid G_l(x,y,t_l)=0\}$, we have that
$\varphi[U_l (t_l)] $ is reduced. Therefore there exists a unique  $a\in\mathfrak {H}$ so that $ (a, b, c) $ solves \eqref{f3}.

It remains to prove that such triples $(a,b,c)$ depend holomorphically on $(h,u,t_l)$. For sufficiently small $\|u-y\|_\infty$ and $(x,y)\in B(0,\delta_1)\times B(0,\delta_2)$, the unique fixed point of $v\mapsto y-(u(x,v)-v)$ is obtained by iterating this map starting with $v=0$. These iterations depend holomorphically on $(x,y)$ and converge uniformly in $\mathfrak {H}$ to a map $v$ such that $(x,y)\mapsto(x,v(x,y))$ is the inverse of $(x,y)\mapsto(x,u(x,y))$. Therefore $v\in\mathfrak {H}$ depends holomorphically on $u$ and $\Tilde{h}(x,y)=h(x,v(x,y))$ depends holomorphically on $u$ and $h$.

For each $l=1,\ldots,L$ the function $\frac{\partial f_l}{\partial y}(x,y)$ has  only one root on $U_l$~\eqref{eq:map b2},  at $ (0, 0) $. We know from Lemma~\ref{l2} that Equivalently, we establish a one-to-one correspondence between isomorphism classes of local $x$-deformations of $(X,\dfx)$ and $x$-isomorphism classes of deformations of the corresponding $V(f_l)$. A function $g\in\Sh{O}_{U_l}$ belongs to the ideal $\langle f_l,\frac{\partial f_l}{\partial y}\rangle$ of $\mathbb{C}[y]\Hat{\otimes}\Sh{O}_{B(0,\delta_1)}$ if and only if $g/\frac{\partial f_l}{\partial y}$ is holomorphic. By Lemma~\ref{l2} this holds if and only if for all $f\in\Sh{O}_{U_l}$ the residue of $gf(\frac{\partial f_l}{\partial y})^{-2}\dfx$ on $X_l$ vanishes. Moreover, for $f\in\frac{\partial f_l}{\partial y}\Sh{O}_{U_l}$ this pairing for $(f,g)$ vanishes. Therefore we obtain the following non-degenerate pairing
\begin{align}\label{eq:pairing}
\Sh{O}_{U_l}\big/\frac{\partial f_l}{\partial y}\Sh{O}_{U_l}\times \Sh{O}_{U_l}\dfx\big/\frac{\partial f_l}{\partial y}\Sh{O}_{U_l}\dfx&\rightarrow\C,&
(f,g)&\mapsto\Res_{U_l} gf(\frac{\partial f_l}{\partial y})^{-2}\dfx.
\end{align}
Using this pairing, the basis $g_{l,1},\ldots,g_{l,{r_l}}$ of $ \Sh{O}_{U_l}\big/\frac{\partial f_l}{\partial y}\Sh{O}_{U_l} $ induces a dual basis  $h_{l,1},\ldots,h_{l,{r_l}}\in\mathbb{C}[y]\Hat{\otimes}\Sh{O}_{B(0,\delta_1)}$ of $ 
\Sh{O}_{U_l}\dfx\big/\frac{\partial f_l}{\partial y}\Sh{O}_{U_l}\dfx$.
% Consequently there exist $h_{l,1},\ldots,h_{l,{r_l}}\in\mathbb{C}[y]\Hat{\otimes}\Sh{O}_{B(0,\delta_1)}$ which induce in the right hand side of~\eqref{eq:pairing}
%$(\Sh{O}_{X_l}\dfx/\frac{\partial f_l}{\partial y})/\Sh{O}_{X_l}\dfx$
%a basis dual to the basis of
%$g_{l,1},\ldots,g_{l,r}$ of the left hand side of~\eqref{eq:pairing}
%%$\Sh{O}_{X_l}/\frac{\partial f_l}{\partial y}\Sh{O}_{X_l}$
%with respect to this pairing. 
The ${r_l}\times {r_l}$ matrix
\begin{align*}%\label{eq: matrix B}
B_{ij} (t_l) =\Res_{U_l(t_l)}\left(g_{l,i}h_{l,j}\dfx\big/\left(\frac{\partial G_l}{\partial y}\right)^2\right)
\end{align*}
depends holomorphically on small $t_l\in\Tm_l$ and stays nearby $\unity$. In particular $B$ is invertible. For small $t_l\in\Tm_l$, the proof of Theorem~\ref{l3} shows that the basis $h_{l,1},\ldots,h_{l,{r_l}}$ of $\Sh{O}_{U_l}\dfx\big/\frac{\partial f_l}{\partial y}\Sh{O}_{U_l}\dfx$ defines also a basis of $\Sh{O}_{U_l(t_l)}\dfx/\frac{\partial f_l}{\partial y}\Sh{O}_{U_l(t_l)}\dfx$.
% that carries over and shows that $h_{l,1},\ldots,h_{l,{r_l}}$ induce also a basis of % the right hand side of~\eqref{eq:pairing}.
% $\Sh{O}_{U_l(t_l)}\dfx/\frac{\partial f_l}{\partial y}\Sh{O}_{U_l(t_l)}\dfx$.
Taking the residue on $U_l(t_l)$, the pairing  \eqref {eq:pairing} remains non-degenerate 
% pairing with %the left hand side of~\eqref{eq:pairing}
% $\Sh{O}_{U_l(t_l)}/\frac{\partial f_l}{\partial y}\Sh{O}_{U_l(t_l)}$
and is represented with respect to the bases $ g_l, h_l$ by the matrix $B (t_l) $. Hence $c$ is given by
\begin{align}\label{eq:formula c}
c_i=\sum_{j=1}^rB^{-1}_{ij}\Res_{U_l(t_l)}\left(\Tilde{h}h_{l,j}\dfx\big/\left(\frac{\partial G_l}{\partial y}\right)^2\right)
\end{align}
and depends holomorphically on $h$, $u$ and $t_l$.

%Given $x\in B(0,\delta_1)$, for sufficiently small $\delta_1$, $|t_l|$ and $\|u-y\|_\infty$, the polynomial $y\mapsto G_l(x,u(x,y),t_l)$ has exactly $ d_l $  roots in $y\in B(0,\delta_2)$. These are also the roots of the Weierstra{\ss} polynomial $y\mapsto F(x,y)$.
For $x\in B(0,\delta_1)$ the roots of $y\mapsto G_l(x,u(x,y),t_l)$ are equal to the roots of the Weierstra{\ss} polynomial $y\mapsto F(x,y)$.
For $n=0,\ldots, d_l -1$ the sums of the $n$-th powers of these roots are equal to
\begin{align*}
\Res_{y\in B(0,\delta_2)}
\frac{\partial G_l}{\partial u}(x,u(x,y),t_l)\frac{\partial u}{\partial y}(x,y)y^ndy\big/G_l(x,u(x,y),t_l).
\end{align*}
For the same $x\in B(0,\delta_1)$, the coefficients of the Weierstra{\ss} polynomials $F$ are  polynomials with respect to these residues~\cite[(2.14)]{Macdonald:79}. In particular, these coefficients depend holomorphically on $(x,u,t_l)\in B(0,\delta_1)\times\mathfrak {H}_l\times\Tm_l$.

We now show that the coefficients of $b$ depend holomorphically on $(x,h,u,t_l)\in B(0,\delta_1)\times\mathfrak {H}\times\mathfrak {H}_l\times\Tm_l$. For any function $g:\mathbb{C}\to\mathbb{C}$, the right hand side of~\eqref{eq:pol 1} defines a polynomial of degree $d-1$,  which takes the same values as $g$ at the roots $y_1,\ldots,y_d$ of $f(y)$. Hence for such $g$, the difference of the left hand side and the right hand side of~\eqref{eq:pol 1} is always a multiple of $f$. We claim that the coefficients of the right hand side are polynomials in the coefficients of the polynomials $f$ and in the expression
\begin{equation}\label{eq:sum}
\sum_{i=1}^dg(y_i)y_i^n\big/\frac{\partial f}{\partial y}(y_i)\quad
0\le n\le d-1.
\end{equation}
In order to prove the claim we may assume that $g(y)$ is a polynomial of degree $d-1$. Due to~\eqref{eq:pol 2} the highest coefficient, i.e. that of $y^{d-1}$, is equal to~\eqref{eq:sum} for $n=0$. Inductively the difference of $y^ng(y)$ and a multiple of $f(y)$ has degree less than $d$. Consequently the coefficient of $y^{d-1-n}$ in $ g $ is equal to the difference of~\eqref{eq:sum} and a polynomial with respect to the coefficients of $f$ and higher coefficients of $g$. This proves the claim.

Given $x\in B(0,\delta_1)$ we apply this claim to the function $g(y)$ given by~\eqref{eq:b}. The corresponding expression~\eqref{eq:sum} is equal to
$$\Res_{y\in B(0,\delta_2)}\left(h(x,y)-c\cdot g_l(x,u(x,y))\right)\frac{\partial u}{\partial y}(x,y)y^ndy\big/G_l(x,u(x,y),t_l).$$
Hence the coefficients of $F$ and of $b$ depend holomorphically on $(x,h,u,t_l)\in B(0,\delta_1)\times\mathfrak{H}\times\mathfrak{H}_l\times\Tm_l$.

Finally, given $(u,h,t_l)$,  $c$ and $b$, $a(x,y)$ is equal to
$$\left(h(x,y)-b(x,y)\frac{\partial G_l}{\partial u}(x,u(x,y),t_l)+c\cdot g_l(x,u(x,y))\right)\big/G_l(x,u(x,y),t_l).$$
By choice of $b$ and since $U_l(t_l)$ (as defined in \eqref{eq:deformation}) is reduced, this expression belongs to $\mathfrak{H}$ and depends holomorphically on $(h,u,t_l)\in\mathfrak {H}\times\mathfrak{H}_l\times\Tm_l$.

Finally, we choose $\epsilon_1>0$ and $\epsilon_2>0$ such that the $c$ in~\eqref{eq:unique c} and~\eqref{eq:formula c} is unique and depends holomorphically on $(h,u,t_l)$, \eqref{eq U subset ball} holds and the $g_l$ form a basis of $\Sh{O}_{U_l(t_l)}\big/\frac{\partial G_l}{\partial y}(x,y,t_l)\Sh{O}_{U_l(t_l)}$ for $|t_l|<\epsilon_1$ and $\|u-y\|_\infty<\epsilon_2$.
\end{proof}
%%%%%%%%%%%%%%%%%%%%%%%%%%%%%%%%%%%%%%%%%%%%%%%%%%%%%%%%%%%%%%%%%%
%%%%%%%%%% existence
%%%%%%%%%%%%%%%%%%%%%%%%%%%%%%%%%%%%%%%%%%%%%%%%%%%%%%%%%%%%%%%%%%
\section {Existence of a universal deformation}\label{sec:existence}
In this section we shall prove our main theorem.

\begin{theorem}\label{t4}
Let  $ g_1,\ldots, g_L $ be a basis for the isomorphism classes of infinitesimal $ x $-deformations of  a locally planar pair 
$ (X, d x) $ and 
\begin{align}\tag{\ref{eq:universal representative}}
(X,\dfx)\hookrightarrow(Z,\dfx_Z)\twoheadrightarrow\Tm
\end{align}
the corresponding $ x $-deformation, constructed in Definition~\ref{universal1}.
Then the local $ x $-deformation
\begin{align}\tag{\ref{eq:universal 1}}
(X,\dfx)\hookrightarrow(Z,\dfx_Z)\twoheadrightarrow\Tm_0
\end{align}
is a universal object in the category of local $ x $-deformations under morphism. This means that any local $x$-deformation 
$(X,\dfx)\hookrightarrow(Y, d x_Y)\twoheadrightarrow\mathcal S_0 $
has a representative
$(X,\dfx)\hookrightarrow(Y, d x_Y)\twoheadrightarrow\mathcal S $
and a unique holomorphic map $\phi:\Sm\to\Tm$ which pulls back~\eqref{eq:universal representative} to a $x$-deformation isomorphic to that representative.%$(X,\dfx)\hookrightarrow(Y, d x_Y)\twoheadrightarrow\mathcal S $. 
\end{theorem}
 In particular,  the isomorphism class of   the local $ x $-deformation \eqref {eq:universal 1}  does
not depend on the choice of the basis $g_l$ for the infinitesimal $ x $-deformations of  $ (X, \dfx) $. This includes that it does not depend upon the choice of local germs $ y_l $ and the corresponding $f_l\in\C\{x, y\} $ such that $ (x_l, y_l) $ maps $ X_l $ biregularly onto $ V (f_l) $.

%Theorem~\ref{t4}, namely that the local deformation constructed in \eqref {eq:universal 1} is a universal. 

\begin{proof}
Let $ (X, d x) $ be locally planar, and adopt the notation of Definition~\ref{def:locally planar}.
Let $(X,\dfx)\hookrightarrow(Y, d x_Y)\twoheadrightarrow\mathcal S_0 $
be a local $ x $-deformation of  $ (X, \dfx) $; it was shown in Theorem~\ref{equivalent deformation} that this
% the corresponding local $ x $-deformation $(X,\dfx)\hookrightarrow(Y, d x_Y)\twoheadrightarrow\mathcal S_0 $ 
is
determined up to isomorphism by $x$-deformations
 $ V (f_l)\hookrightarrow V (F_l)\twoheadrightarrow\Sm_0 $
% $ (X _q, x_q, y_q)\hookrightarrow (Y _q, x_{Y , q}, y_{Y , q})\twoheadrightarrow\Sm_0 $
of germs $ V (f_l)\simeq X_l $ at the points $q_1,\ldots,q_L$ where $x_l$ is not a local coordinate. % These points contain the singularities of $X$ together with the smooth points of $X$ which are roots of $\dfx$. 
It was argued in Claim 1 of Theorem~\ref{equivalent deformation} that we may without loss of generality take a representative deformation whose base $\mathcal S $ is an open ball around $ 0\in\mathbb C ^ J $, where $J$ is the embedding dimension of $\mathcal S$.
%These local deformations can be represented as complex space germs
%\begin{align}\label{f1a}
%Y_l&=\{(x,y,s)\in B(0,\delta_1)\times B(0,\delta_2)\times\Sm\mid F_l(x,y,s)=0\}
%\end{align}
%with $F_l\in\mathbb{C}\{x,y\}\Hat{\otimes}\Sh{O}_{\Sm_0}$. The germs $F_l\in\mathbb{C}\{x,y\}\Hat{\otimes}\Sh%{O}_{\Sm_0}$ extend to germs $\Tilde{F}_l\in\mathbb{C}\{x,y\}\Hat{\otimes}\Sh{O}_{\mathbb{C}^k_0}$. Due to %Theorem ~\ref{equivalent deformation} the whole deformation extends to a deformation over an open neighbourhood %of $0$ in $\mathbb{C}^k$. Consequently it suffices to prove the theorem for deformations, whose base $\Sm$ is %an open ball of $0$ in $\mathbb{C}^k$. 
%
%and assume without loss of generality that the complex space $\Sm$ is embedded in some $\mathbb{C}^k$ with %marked point $0\in\Sm$.
Furthermore, the deformations of the space germs may be represented as
\begin{align*} %\label{f1a}
V (F_l) &=\{(x,y,s)\in B(0,\delta_1)\times B(0,\delta_2)\times\Sm\mid F_l(x,y,s)=0\},
\end{align*}
where each $F_l(x,y,s)$ is  a Weierstra{\ss} polynomial of degree $ d_l $  in $\mathbb{C}\{x\}[y]\Hat{\otimes}\Sh{O}_{\Sm}$, and $f_l(x,y) = F_l(x,y,0)$.

%with $F_l\in\mathbb{C}\{x,y\}\Hat{\otimes}\Sh{O}_{\Sm_0}$ and  by the Weierstra{\ss} Preparation Theorem \cite[Chapter I Theorem 1.4]{PR:94} we may assume that each $F_l(x,y,s)$ is  a Weierstra{\ss} polynomial of degree $ d_l $  in $\mathbb{C}\{x\}[y]\Hat{\otimes}\Sh{O}_{\Sm}$.

The theorem calls for a holomorphic map $\phi:\Sm\to\Tm$ such that the pullback of  \eqref {eq:universal representative} is isomorphic to
the given $ x $-deformation. % ~\eqref{eq:deformation 1}. 
By Theorem~\ref{equivalent deformation}, this is equivalent to the construction of holomorphic maps $\phi=(\phi_1,\ldots,\phi_L):\Sm\to\Tm=\Tm_1\times\ldots\times\Tm_L$ such that for all $l=1,\ldots,L$ the pullback of $X_l\hookrightarrow Z_l\twoheadrightarrow\Tm$ with respect to $\phi_l$ is $x$-isomorphic to $X_l\hookrightarrow Y_l\twoheadrightarrow\Sm$. As in~\eqref{eq:mor}, alongside the maps $\phi_1,\ldots,\phi_L$ we shall construct units $H_1,\ldots,H_L$ and holomorphic functions $u_1,\ldots,u_L$ such that for all $l=1,\ldots,L$ we have
\begin{equation}\label{f1}
H_l(x,y,s)G_l(x,u_l(x,y,s),\phi_l(s))=F_l(x,y,s)
\end{equation}
for all $s\in\Sm$. These maps will define an $x$-morphism $(x,y,s)\mapsto(x,u_l(x,y,s), \phi_l(s))$ from the deformation $X_l\hookrightarrow Y_l\twoheadrightarrow\Sm$ to the pullback of $X_l\hookrightarrow Z_l\twoheadrightarrow\Tm$ with respect to $\phi_l$. In fact, we will show that they define an $x$-isomorphism.

We now fix $s\in \mathcal{S}$ and 
consider the straight line from $0$ to $s$ parameterised by $\lambda s$ with $\lambda \in [0,1]$.
The derivative of~\eqref{f1} with respect to $\lambda$ gives
\begin{align}\label{f2}
\frac{\mathrm{d} H_l}{\mathrm{d} \lambda}&(x,y,\lambda s)G_l(x,u_l(x,y,\lambda s),\phi_l(\lambda s))\nonumber\\
&+H_l(x,y,\lambda s)\frac{\partial G_l}{\partial u_l}(x,u_l(x,y,\lambda s),\phi_l(\lambda s))\frac{\mathrm{d} u_l}{\mathrm{d} \lambda}(x,y,\lambda s)\\
&+H_l(x,y,\lambda s)\frac{\mathrm{d} \phi_l}{\mathrm{d} \lambda}(\lambda s)\cdot g_l(x,u_l(x,y,\lambda s))\hspace{20mm}=\frac{\mathrm{d} F_l}{\mathrm{d} \lambda }(x,y,\lambda s).\hspace{-10mm}\nonumber
\end{align}

Applying Lemma~\ref{l5} for $l=1,\ldots, L$ and $\lambda \in [0,1]$ to
\begin{align*}
h(x,y)&:=H_l^{-1}(x,y,\lambda s)\frac{\mathrm{d} F_l}{\mathrm{d}\lambda}(x,y,\lambda s),&
u(x,y)&:=u_l(x,y,\lambda s),&t_l&:=\phi_{l}(\lambda s) 
\end{align*}
defines $(a_{\lambda s},b_{\lambda s},c_{\lambda s}) \in \mathfrak{H} \times \mathfrak{H}_l \times \mathbb{C}^{r_l}$ as in Equation~\eqref{f3} which depend holomorphically on $\lambda s$. We now see from \eqref{f2} that in order to determine the $x$-morphism \eqref{f1} at $s$ it suffices to solve the equations
\begin{align}\label{f4}
\frac{\mathrm{d} H_l}{\mathrm{d} \lambda}(x,y, \lambda s)&=a_{\lambda s}(x,y)H_l(x,y,\lambda s),&
\frac{\mathrm{d} u_l}{\mathrm{d} \lambda}(x,y,\lambda s)&=b_{\lambda s}(x,y),&
\frac{\mathrm{d} \phi_l}{\mathrm{d} \lambda}(\lambda s)&=c_{\lambda s}
\end{align}
at $\lambda=1$.
These equations define a holomorphic vector field along the straight line connecting $0$ with $s$. For each sufficiently small $s\in\Sm$ we apply the Picard-Lindel\"of Theorem and integrate the vector fields \eqref{f4} along these straight lines up to $s$, hence obtaining at $s$ a solution of \eqref{f1}. Because the coefficients $a_{\lambda s}, b_{\lambda s}, c_{\lambda s}$ of \eqref{f4} depend holomorphically on $s$, so do the solutions $(H_l, u_l, \phi_l)$. Hence we have constructed an $x$-morphism $(x,y,s)\mapsto(x,u_l(x,y,s), \phi_l(s))$ as in \eqref{f1}. 

We now show that after shrinking $\Sm$, the pullback of~\eqref{eq:deformation} under $\phi_l$ is $x$-isomorphic to $V(f_l)\hookrightarrow V(F_l)\twoheadrightarrow\Sm$.
We again fix $s \in \mathcal{S}$ and show that if 
% the supremum norm $\|u_l-y\|_\infty$ 
$s$ is sufficiently small, then there exists 
$v_l \in\mathbb{C}\{x,y\}$ such that 
% $v_l(x,u_l(x,y,s))=y$ and that 
$(x,y)\mapsto(x,v_l(x,y))$ is the inverse of $(x,y)\mapsto(x,u_l(x,y,s))$. 
% Since the map $(x,y)\mapsto(x,u_l(x,y,s))$ is determined only on $Y_l$ due to the 
The Weierstra{\ss} Division Theorem \cite[Chapter I Theorem 1.1]{PR:94} proves that the map $(x,y)\mapsto(x,u_l(x,y,s))$ can be represented uniquely as a polynomial with respect to $y$ of degree less than $ d_l $ . We now use the notation introduced in the paragraph above Lemma~\ref{l5}. If $s$ is sufficiently small, then $|\frac{\partial u_l}{\partial y}-1|$ is bounded on $(x,y)\in B(0,\delta_1)\times B(0,\delta_2)$ by $\frac{1}{2}$. Hence for all fixed $x\in B(0,\delta_1)$ the map $y\mapsto u_l(x,y,s)$ has bounded inverse $y\mapsto v_l(x,y)$ from $B(0,\delta_2)$ onto an open subset of $\mathbb{C}$. Moreover the inverse function $v_l$ depends holomorphically on $s$.

%For sufficiently small $s$, the corresponding $u_1,\ldots, u_L$ have small distance $\|u_l-y_l\|_\infty$ from $ y_l $, so $(x_l,y_l)\mapsto(x_l,u_l(x_l,y_l,s))$ are $x$-isomorphisms. Consequently, on a neighbourhood of $0$ the given $ x $-deformation  is isomorphic to the pullback of  \eqref {eq:universal representative} with respect to $\phi$.

It remains only to prove the uniqueness of the map $\phi $. Any map $\phi:\Sm\to\Tm$ with respect to which the pullback of  \eqref {eq:universal representative} is isomorphic to the given $ x $-deformation is described by solutions $(H_l,u_l,\phi_l)_{l\in\{1,\ldots,L\}}$ of~\eqref{f1} obeying \eqref{f2}. Hence uniqueness follows from Lemma~\ref{l5} and the Picard-Lindel\"of Theorem.
\end{proof}

\begin{remark}
Suppose that the special fibre $(X,dx)$ is real with respect to an anti-holomorphic involution $\eta$ (Definition~\ref{definition:single}(c)). At the end of Section~\ref{sec:infinitesimal} we showed that we can choose the basis $g_1,\dotsc,g_L$ so that $\eta$ lifts to an involution of the deformation~\eqref{eq:universal representative}. Furthermore, the restriction of the corresponding local $x$-deformation $(X,\dfx)\hookrightarrow(Z,\dfx_Z)\twoheadrightarrow\Tm_0$~\eqref{eq:universal 1} to the fixed point set of the action of $\eta$ on $\Tm_0$ is universal in the category of local $x$-deformations of $(X,dx)$ that are real with respect to $\eta$. 

Indeed, given an $x$-deformation $(X,\dfx)\hookrightarrow(Y, d x_Y)\twoheadrightarrow\mathcal S$, Theorem~\ref{t4} implies that the morphism $\phi$ commutes with $\eta$. Furthermore $\phi$ maps the fixed point set of $\eta$ in $\Sm$ into the fixed point set of $\eta$ in $\Tm$. %In particular, the restriction of~\eqref{eq:universal representative} to the fixed point set of $\eta$ in $\Tm$ defines a universal local $x$-deformation of locally planar $(X,\dfx)$ obeying the reality condition.
\end{remark}
\begin{corollary}\label {corollary:universal}
Let $ (X, d x) $ be a locally planar pair and $ g_1, \ldots , g_L $ a basis for the isomorphism classes of infinitesimal $ x $-deformations of $ (X, d x) $. We consider the corresponding $x$-deformation $(X,\dfx)\hookrightarrow(Z, \dfx_Z)\xtwoheadrightarrow{\pi}\Tm$ constructed in \eqref{eq:universal representative}.
Then there exists an open neighbourhood $\mathcal{T}'$ of $0$ in $\mathcal{T}$, so that for every $t \in \mathcal{T}'$ the local $x$-deformation 
$$ (\pi^{-1}[t], dx_Z|_{\pi^{-1}[t]}) \hookrightarrow (Z,dx_Z) \xtwoheadrightarrow{\pi} \mathcal{T}'_{t} $$
is a universal local $x$-deformation of the fibre $\pi^{-1}[t]$. 
% \xtwoheadrightarrow{\pi}
\end{corollary}
\begin{proof}
By Theorem~\ref{l3}(ii), the isomorphism classes of infinitesimal $x$-deformations of the fibres of $\pi$ comprise a vector bundle over $\mathcal{T}$. We take any open neighbourhood $\mathcal{T}'$ of $0$ in $\mathcal{T}$ over which this vector bundle is trivial. Then the basis $g_1,\dotsc,g_L$ of the fibre at $t=0$ can be extended to linearly independent sections $g_1,\dotsc,g_L$ over $\mathcal{T}'$. At each $t\in \mathcal{T}'$, these sections give a basis for the isomorphism classes of infinitesimal $ x $-deformations of $\pi^{-1}[t]$.
Clearly then by Theorem~\ref{t4} the restriction of the $ x $-deformation $(Z, \dfx_Z)\twoheadrightarrow\Tm'$ represents the local universal $ x $-deformation of all its fibres induced by $g_1,\dotsc,g_L$. 
\end{proof}
We shall henceforth take $\Tm $ to be this reduced neighbourhood, so that it enjoys the property of Corollary~\ref {corollary:universal}.

% We find it convenient below to extend Definitions~\ref {def2} and \ref {def3} in the obvious way to encompass germs of locally planar pairs $ (V (f), \dfx) $; the isomorphism classes of these pairs are clearly the same as the $x $-isomorphism classes of germs $ V(f) $.

 It is a general property of universal deformations that whenever the special fibre $ (X, \dfx) $ possesses a non-trivial automorphism,
 the base space $\Tm $ of the universal deformation does not locally parameterise isomorphism classes of the fibres,  as the automorphism extends to an isomorphism of fibres over distinct elements of the base $\Tm $. We illustrate this in the present situation with an example.
 
%  Example~\ref {e1}),We now turn to the question of whether $\Tm$ locally parameterises the isomorphism classes of locally planar $(X,\dfx)$ with prescribed poles (Definition~\ref{definition:single}). In fact it does not, as is easily revealed by considering the case when the special fibre has a non-trivial automorphism. Since the $ x $-deformation $(X,\dfx)\hookrightarrow(Z,\dfx_Z)\twoheadrightarrow\Tm$ is universal, the automorphism of $(X,\dfx)$ induces automorphisms of $(Z,\dfx_Z)$ and $\Tm $, with respect to which the above maps are equivariant. The fibres over distinct points of the base space are mapped to one another under the automorphism of the total space and hence are isomorphic. We present now an example which demonstrates how a symmetry in $f_l$ induces an isomorphism between the fibres over different elements of $\Tm$. 

% Using Theorem~\ref{equivalent deformation}, the universal deformation is described by the corresponding universal deformations of space germs $V(f_l)\hookrightarrow V(G_l)\twoheadrightarrow\Tm_l$ for $l=1,\ldots,L$. 
% TODO change the deformation to a deformation of compact spaces.
\begin{Example}\label{e1}
Let $X$ be the projective curve $\{(X_0:X_1:X_2)\in\mathbb{P}^2\mid X_2^3=X_1^2X_0\}$ and consider the meromorphic function $x=X_1/X_0$ on $X$, whose only pole is at $p_1=(0:1:0)$ and is of order $M_1=3$. In this example there is only one point $q_1=(1:0:0)$ on $X^\circ$ at which $x$ fails to be a local coordinate. Writing $y$ for the germ of the meromorphic function $X_2/X_0$ at $q_1$, the pair $(x,y)$ maps the germ $X_{q_1}$ biregularly onto $V(f)$ with $f(x,y)=y^3-x^2\in\mathbb{C}\{x,y\}$.

Using Theorem~\ref {l3}, % Lemma~\ref {l1}, 
the isomorphism classes of infinitesimal $x$-deforma\-tions of $ (X,\dfx) $ are parameterised by ${\C\{x, y\}}/{\langle f, f_y\rangle} $~\eqref{eq:quotient}, for which $1,x,y,xy$ provides a basis. The corresponding $x$-deformation $(X,\dfx)\hookrightarrow(Z,\dfx_Z)\twoheadrightarrow\Tm $ is given by
$$G(x,y,t_1,t_2,t_3,t_4)=y^3-x^2+t_1+t_2x+t_3y+t_4xy . $$
As shown above, this $x$-deformation represents a universal local $x$-deformation.
Writing $q=e^{\frac{2}{3}\pi\ci}$ and using the notation of  \eqref {eq:mor},  the pairs $(H,u)=(1,qy)$ and $(H,u)=(1,q^2y)$ each define isomorphisms of $ (X,\dfx) $. They lift to  % $\dfx$ preserving 
isomorphisms between the fibres of $ (Z,\dfx_Z)\twoheadrightarrow\Tm $ over $(t_1,t_2,t_3,t_4)$, $(t_1,t_2,qt_3,qt_4)$ and $(t_1,t_2,q^2t_3,q^2t_4)$. 
% Hence we see that in this example $\mathcal T $ does not serve as a parameter space for the local deformations of $ (V (f),\dfx) $. 
 Note that the presence of these isomorphisms does not contradict the uniqueness property of the universal local deformation: this uniqueness only excludes non-trivial isomorphisms of $(Z,\dfx_Z)\twoheadrightarrow\mathcal T $ which act trivially on the fibre $(X,\dfx)$ over $0\in\Tm$. However, the isomorphisms we have constructed do not restrict to the identity automorphism of $(X,\dfx)$.
%\begin {align*}
% V (f) &\rightarrow & Z &\rightarrow &\mathcal T\\
%\verteq & &\downarrow &\downarrow\\
% V (f) &\rightarrow & Z &\rightarrow &\mathcal T
%\end {align*}

% % % PUT THIS LATER
In this example, the isomorphism classes of locally planar pairs $(X,d x) $ in a neighbourhood of $0\in\Tm$ are locally parameterised not by $\Tm $ but by $\Tm/\mathbb Z_3 $, giving locally the structure of an orbifold whose
finite group action is the induced action of the automorphism group of the central fibre. %This will be made explicit later in this section by the construction of a 3-fold cover $ \dis:\Tm\rightarrow\C ^ 4\cong  {\C\{x, y\}} /{\langle f, f_y\rangle} $, where this map is the symmetrisation of the branch values of $x $.
\end{Example}
%
%TODO: Put either here or in the Introduction (preferably): A universal deformation gives us a space on which we can consider Whitham vector fields. 
%
%
%%%%%%%%%%%%%%%%%%%%%%%%%%%%%%%%%%%%%%%%%%%%%%%%%%%%%%%%%%%%%%%%%%
%%%%%%%%%% hyperelliptic
%%%% with prescribed poles, comprised of a curve and a single differential satisfying Definition~\ref{definition:single}. In a forthcoming work we shall extend our results to triples $ (X, \dfx, d y) $ which satisfy a natural analogue of Definition~\ref{definition:single}, namely that poles of $dy $ are also prescribed and that the locally planar structure is compatible with both $\dfx $ and $dy $. There is a particular case, arising in a number of integrable systems, for which the existence of a local universal deformation of such triples $ (X, \dfx, d y) $ follows quite easily from our results above. Namely, we consider in this section the case of locally planar pairs $ (X, \dfx) $ where $ X $ possesses a  global meromorphic function $y$ of degree two, and hence is hyperelliptic. Instead of directly requiring the existence of the global function $y$ of degree two, we impose an involution $\sigma $ upon  $(X,\dfx) $ and the local functions $ y_q $ of Definition~\ref{definition:single}, such that the quotient $X/\sigma$ has arithmetic genus zero. 
\section{Deformations of hyperelliptic curves with one differential}\label{sec:def1a}
Hitherto we have considered deformations of locally planar pairs $ (X, \dfx) $ with prescribed poles. %, comprised of a curve and a single differential satisfying Definitions~\ref{definition:single}-\ref{def:locally planar}. 
In a forthcoming work we shall extend our results to locally planar triples $ (X, \dfx, d y) $. By locally planar triples we mean that the poles of $dy $ are also prescribed and that local primitives $x,y $ provide a local embedding of the curve into $\mathbb C ^ 2 $.  There is a particular case, arising in a number of integrable systems, for which the existence of a local universal deformation of such triples $ (X, \dfx, d y) $ follows quite easily from our results above. Namely, we consider in this section the case of locally planar pairs $ (X, \dfx) $ where $ X $ possesses a  global meromorphic function $y$ of degree two, and hence is hyperelliptic. Instead of directly requiring the existence of the global function $y$ of degree two, we consider the more general situation when $X$ has a holomorphic involution $\sigma $ such that $\sigma^\ast \dfx=-\dfx$ and the quotient $X/\sigma$ is smooth. Then the curve $X$ is hyperelliptic if $X/\sigma \simeq \mathbb{P}^1$. We further impose that the local functions $y_q$ of Definition~\ref{def:locally planar} are local parameters of this quotient:
\begin{definition}\label{def:triple}
A \emph{symmetric triple} $(X,\dfx,\sigma)$ is a locally planar pair $(X,\dfx)$ together with an involution $\sigma$ of $X$ such that
\begin{enumerate}
\item[(i)] $\sigma^\ast\dfx=-\dfx$
\item[(ii)] For each $l=1,\ldots,L$ the germ $y_l$ of Definition~\ref{def:locally planar} is a local parameter of $X/\sigma$ at $q_l$.
\end{enumerate}
The symmetric triple is called \emph{hyperelliptic}, if in addition $X/\sigma$ is biholomorphic to $\mathbb{P}^1$.
\end{definition}
Since $\sigma$ preserves the singularities, and the poles and smooth roots of $\dfx$, the sets $\{p_1,\ldots,p_K\}$ and $\{q_1,\ldots,q_L\}$ are permuted by $\sigma$. We also denote by $\sigma$ the induced involutions on $\{1,\ldots,K\}$ and $\{1,\ldots, L\}$.
Note that condition~(ii) forces $X/\sigma$ to be smooth, and that at each $q_l$ which is fixed by $\sigma$, the germ $y_l$ obeys $\sigma^\ast y_l=y_l$.

 By Lemma~\ref{lemma:uniqueness} and Theorem~\ref{equivalent deformation}, the isomorphism classes of local $x$-deformations $(X,dx)\hookrightarrow (Y,dx_Y) \twoheadrightarrow\Sm_0$ of a locally planar pair $(X,\dfx)$ are in one-to-one correspondence with the isomorphism classes of $x$-deformations $X_l\hookrightarrow Y_l\twoheadrightarrow\Sm_0$ for $l=1,\ldots,L$. We shall utilise this correspondence to directly define local $x$-deformations $ (X,\dfx, \sigma)\hookrightarrow (Y,\dfx_Y, \sigma)\twoheadrightarrow\Sm_0$, rather than introducing a notion of  $ x$-deformations of symmetric triples $(X,\dfx,\sigma)$. 

If the involution $\sigma$ of the symmetric triple $(X,\dfx,\sigma)$ lifts to an involution of $Y$ covering the trivial involution of $\Sm_0$, then for each $l=1,\ldots,L$ the map $\sigma:X_l\mapsto X_{\sigma l}$ should lift to an commutative diagram
\begin{gather}\label{deformation sigma 1}\begin{array}{lclcl}
X_l&\hookrightarrow&Y_l&\twoheadrightarrow&\Sm_0\\
\downarrow\sigma&&\downarrow\sigma&&\|\\
X_{\sigma l}&\hookrightarrow&Y_{\sigma l}&\twoheadrightarrow&\Sm_0.
\end{array}\end{gather}
For each $l=1,\ldots,L$, choose as in Theorem~\ref{equivalent deformation}, $f_l\in\mathbb{C}\{x,y\}$ and $F_l\in\mathbb{C}\{x,y\}\Hat{\otimes}\Sh{O}_{\Sm_0}$ such that $V(f_l)\hookrightarrow V(F_l)\hookrightarrow\Sm_0$ is $x$-isomorphic to $X_l\hookrightarrow Y_l\twoheadrightarrow\Sm_0$. In order that on each fibre of the deformation, $y_l$ satisfies the condition in Definition~\ref{def:triple},  we assume that these $x$-isomorphisms commute with the maps $\sigma$ in the commutative diagrams~\eqref{deformation sigma 1}-\eqref{deformation sigma 2}:
\begin{gather}\label{deformation sigma 2}\begin{array}{lclcl}
V(f_l)&\hookrightarrow&V(F_l)&\twoheadrightarrow&\Sm_0\\
\downarrow\sigma:(x,y)\mapsto(-x,y)&&\downarrow\sigma:(x,y,s)\mapsto(-x,y,s)&&\|\\
V(f_{\sigma l})&\hookrightarrow&V(F_{\sigma l})&\twoheadrightarrow&\Sm_0.
\end{array}\end{gather}
\begin{definition}\label{def:deformation sigma}
A local $x$-deformation of a symmetric triple $(X,\dfx,\sigma)$ is a local $x$-deformation $(X,\dfx)\hookrightarrow(Y,\dfx_Y)\twoheadrightarrow\Sm_0$ which satisfies:
\begin{enumerate}
\item[(i)] The involution $\sigma$ extends to an involution of $Y$ covering the trivial involution on $\Sm_0$. This extended $\sigma$ satisfies $\sigma^\ast\dfx_Y=-\dfx_Y$ and commutes with the maps $X\hookrightarrow Y\twoheadrightarrow\Sm_0$. In particular for any $l=1,\ldots,L$ we obtain the commutative diagram~\eqref{deformation sigma 1}.
\item[(ii)] For each $l=1,\ldots,L$, the induced local $x$-deformation $(X_l, x_l) \hookrightarrow (Y_l, x_{Y,l}) \twoheadrightarrow \mathcal{S}_0$ is $x$-isomorphic to a local deformation $V(f_l)\hookrightarrow V(F_l)\twoheadrightarrow\Sm_0$ such that the diagrams \eqref{deformation sigma 1} and \eqref{deformation sigma 2} commute, and the $x$-isomorphisms between them comprise a three-dimensional commutative diagram.
\end {enumerate}
\end {definition}

%We could consider more general deformations with non-trivial actions of $\sigma$ on $\Sm$, but only the fibres of $Y\twoheadrightarrow\Sm$ over the fixed points of $\sigma$ in $\Sm$ are preserved by $\sigma$. 

The morphisms of local $x$-deformations of symmetric triples $(X,dx,\sigma)$ are the morphisms of local $x$-deformations of $(X,\dfx)$ %(Definition~\ref{def3})
such that $\psi$ of Definition~\ref{def3} commutes with $\sigma$.
\begin{lemma}
Let $(X,\dfx,\sigma)$ be a symmetric triple. Then for any $l$ with $\sigma l\ne l$, $X_l$ is smooth and we may choose the local parameter $y_l$ of $X/\sigma$ such that $x_l=y_l^{r_l+1}$ and $x_{\sigma l}=-y_{\sigma l}^{r_l+1}$ for some $r_l\in\mathbb{N}$. For $l=\sigma l$ we may choose the local parameter $y_l$ of $X_l/\sigma$ such that $X_l\simeq V(f_l)$ with $f_l(x_l,y_l)=y_l^{r_l+1}-x_l^2$ and $r_l\in\mathbb{N}$.
\end{lemma}
\begin{proof}
If $l\ne\sigma l$, then the two-sheeted covering $X \to  X/\sigma$ is unbranched at $q_l$ and $q_{\sigma l}$, and therefore a local biholomorphism at these points. Hence $X_l$ is smooth at $q_l$ and $q_{\sigma l}$. Since $\dfx=d(x_l)$ has a root at $q_l$ of order $r_l\in\mathbb{N}$ there exists a local parameter of $X_l$ with $x_l=y_l^{r_l+1}$.

If $l=\sigma l$, then $X_l\to X_l/\sigma$ is a two sheeted covering, and $X_l/\sigma$ is smooth. In this case $X_l$ is preserved by $\sigma$ and the local parameter $y_l$ of $X_l/\sigma$ obeys $\sigma^\ast y_l=y_l$. Since $x_l^2$ is a function on $X_l/\sigma$ we may choose the local parameter $y_l$ such that $x_l^2=y_l^{r_l+1}$.
\end{proof}
%As in Definition~\ref{def2} we denote the corresponding deformations by
%\begin{equation}\label{eq:deformation 1a}
%(X,\dfx)\hookrightarrow(Y,\dfx_Y)\twoheadrightarrow\Sm.
%\end{equation}
%Theorem~\ref{equivalent deformation} carries over and shows that a local deformation of hyperelliptic $ (X,\dfx,\sigma)$ is equivalent to a deformation of each $X_l=X_{q_l}$ preserving the involution $\sigma$.
In order to describe the deformations of symmetric triples $(X,\dfx,\sigma)$ we decompose the algebra $\mathbb{C}\{x,y\}$ into a direct sum of symmetric and anti-symmetric elements with respect to the involution $\sigma^\ast x=-x$ and $\sigma^\ast y=y$:
\begin{align*}
g&=g^++g^-&g^+&=\tfrac{1}{2}(g+\sigma^\ast g)&g^-&=\tfrac{1}{2}(g-\sigma^\ast g).
\end{align*}
We choose a basis for the isomorphism classes of infinitesimal deformations of the symmetric triple $(X,dx,\sigma)$ as follows: 
For $\sigma l=l$ we choose a basis of the symmetric part of $\mathbb{C}\{x,y\}/\langle f_l,\frac{\partial f_l}{\partial y}\rangle$. 
For $\sigma l \neq l$ we require that 
%the basis in~\eqref{eq:deformation} to be of the form
\begin{align*}
(g_{l,1},\ldots,g_{l,r})=(g_{\sigma l,1},\ldots,g_{\sigma l,r})&=(1,y_l,\ldots,y_l^{r-1}).
\end{align*}
Analogously to Definition~\ref{universal1}, this basis again defines a local $x$-deformation denoted by
$(X,\dfx)\hookrightarrow(Z_\sigma,\dfx_Z)\twoheadrightarrow\Tm_{0,\sigma}$,
whose isomorphism class 
is independent of the choice of $\sigma $-symmetric basis $ g_l $. This agrees with the restriction of~\eqref{eq:universal 1} to the fixed point set $\Tm_{0,\sigma}$ of the action of $\sigma$ on $\Tm_0$.
This local $x$-deformation satisfies the conditions of Definition~\ref{def:deformation sigma} and hence is a local $x$-deformation
\begin{equation}\label{eq:universal 1a}
	(X,\dfx,\sigma)\hookrightarrow(Z_\sigma,\dfx_Z, \sigma)\twoheadrightarrow\Tm_{0,\sigma}
\end{equation}
of the symmetric triple.

Our arguments carry over to the $\sigma$-symmetric case, yielding the following analogue to Theorem~\ref{t4}:
\begin{Corollary}\label{universal 2a}
Let $(X,\dfx,\sigma)$ be a symmetric triple as in Definition~\ref{def:triple}. Then for each basis $g_l$ as above, \eqref{eq:universal 1a} defines a universal local $x$-deformation %$(X,\dfx,\sigma)\hookrightarrow(Z,\dfx_Z,\sigma)\twoheadrightarrow\Tm_0$
in the category of local $x$-deformations of $(X,\dfx,\sigma)$ defined in Definition~\ref{def:deformation sigma}. % the deformations~\eqref{eq:deformation 1a}.\qed
\end{Corollary}
If $(X,\dfx,\sigma)$ is hyperelliptic, then the arithmetic genus of the quotient of the fibres by $\sigma$ is constant along a deformation. This means that all fibres are hyperelliptic with hyperelliptic involution $\sigma$.

If we additionally require $(X,\dfx,\sigma)$ to be real with respect to an anti-holomorphic involution $\eta$, as in Definition~\ref {definition:single}(c),  then $\eta$ lifts to an involution of~\eqref{eq:universal 1a}. By restricting to the fixed point set $\Tm_{0,\sigma,\mathbb{R}}$ of $\eta$ acting on $\Tm_{0,\sigma}$, \eqref{eq:universal 1a} describes a universal object in the category of local $x$-deformations of hyperelliptic triples $(X,dx,\sigma)$, where the deformations are required to be real with respect to $\eta$. 
% Obeying the conditions~(A1)-(B1),(D1) and (E1).
In particular this includes the simply periodic solutions of the KdV,  % ( compare Section~\ref{sec:kdv}) 
sinh-Gordon, sine-Gordon and NLS equations. % (compare Section~\ref{sec:sinh} and \cite{H-K-S-3}). 
In \cite{G-S-1} the space of spectral curves is investigated as a covering with respect to the local parameters $t\in\Tm$. For real periodic solutions of the KdV equation, $t\in \mathcal{T}$ provides a global parameter of the space of spectral curves \cite{MO:75,Korotyaev:08}.
In \cite{H-K-S-3} these deformations are applied to the periodic solutions of the sinh-Gordon equation. We shall discuss these and other examples in more detail in forthcoming work.
\bibliographystyle{alpha}
%\bibliography{ref}%harmonic}
\bibliography{harmonic}
\end{document}